\documentclass[10pt]{amsart}
\usepackage{amsthm, amsmath,amssymb, graphicx, url, color}
\usepackage[margin=1in]{geometry}
\usepackage{multirow}

\DeclareMathOperator{\coker}{coker}

\DeclareMathOperator{\Tr}{Tr}

\DeclareMathOperator{\Hom}{Hom}
\DeclareMathOperator{\Gr}{Gr}
\DeclareMathOperator{\Ext}{Ext}

\DeclareMathOperator{\Imm}{Im}
\DeclareMathOperator{\KZ}{\mathtt{KZ}}
\DeclareMathOperator{\HH}{\mathcal{H}}
\DeclareMathOperator{\DD}{\mathcal{D}}
\DeclareMathOperator{\AAA}{\mathcal{A}}
\DeclareMathOperator{\hh}{\mathbf{h}}
\DeclareMathOperator{\grade}{\mathsf{grade}}
\newcommand{\into}{\hookrightarrow}
\newcommand{\onto}{\twoheadrightarrow}

\def\CC{\mathbb{C}}
\def\NN{\mathbb{N}}
\def\QQ{\mathbb{Q}}

\def\ZZ{\mathbb{Z}}
\def\PP{\mathbb{P}}
\def\OO{\mathcal{O}}

\newtheorem{theorem}{Theorem}[section]
\newtheorem{lemma}[theorem]{Lemma}
\newtheorem{prop}[theorem]{Proposition}
\newtheorem{corr}[theorem]{Corollary}

\theoremstyle{definition}

\newtheorem*{remark}{Remark}

\begin{document}

\title{Finite dimensional representations of the rational Cherednik algebra for $G_4$}
\author{Yi Sun}
\address{206 Cabot Mail Center, 60 Linnaean St., Cambridge, MA 02138.}
\email{yisun@fas.harvard.edu}

\date{\today}
\begin{abstract}
In this paper, we study representations of the rational Cherednik algebra associated to the complex reflection group $G_4$.  In particular, we classify the irreducible finite dimensional representations and compute their characters.
\end{abstract}

\maketitle

\section{Introduction}

For a complex reflection group $W$ with reflection representation $V$, we may define the rational Cherednik algebra $H_c(W)$ to be a deformation of the algebra $\CC[W] \ltimes \mathcal{D}(V)$ dependent on parameters $c_s$ for each conjugacy class of reflections $s$ in $W$. Some recent study has focused on the classification of representations of $H_c(W)$ for various groups $W$.  The basic question in this case is to determine the characters of the irreducible $H_c(W)$-modules in a certain category $\OO$ for all values of the parameter $c$, and in particular to determine which of them are finite dimensional and to find their dimensions.  For groups of type $A$, this question was studied in \cite{BEG} and \cite{CE}, and for dihedral groups, it was addressed in \cite{Chm}.

In this paper, we classify and give character formulas for the finite dimensional irreducible representations for $W = G_4$, the first sporadic complex reflection group in the Shephard-Todd classification.  This case is especially interesting because it is the only one other than wreath products for which the rational Cherednik algebra provides a quantization of a smooth symplectic variety (see \cite{Bel}). Our methods should be easily applicable to other two dimensional complex reflection groups.

The structure of $H_c(G_4)$ depends on a two-dimensional parameter $c = (c_1, c_2)$.  In an appropriately chosen coordinate system, we find that there are finite dimensional representations only on a periodic family of lines forming a wallpaper pattern of type $*632$.  For points on exactly one line in this pattern, there is a unique irreducible finite dimensional representation, and the category of finite dimensional representations is semi-simple.  This classification is given in Theorem \ref{lines}.  At intersections of the lines, the representations from each line may degenerate.  We find that there are three types of intersection points and that the structure of the finite dimensional irreducible representations depends on the type of the point. We give a complete classification of such representations in Theorem \ref{points}.

The remainder of this paper is structured as follows.  In Section 2, we give some basic definitions and preliminaries about the rational Cherednik algebra.  In Section 3, we state the main theorems.  In Section 4, we collect some results which will be necessary for the proofs.  In Section 5, we give proofs of the results.

\section{Background}

In this section, we define the rational Cherednik algebra and introduce some standard objects relating to it.  We then specialize to the case $W = G_4$ and give the specific notations we will use.

\subsection{The rational Cherednik algebra}

Let $W$ be a complex reflection group with reflection representation $V$.  Let $S$ be the set of complex reflections in $W$; for each $s \in S$, let $\alpha_s \in V^*$ and $\alpha_s^\vee \in V$ be eigenvectors with nontrivial eigenvalue, normalized so that $(\alpha_s, \alpha_s^\vee) = 2$.  Let $\lambda_s$ be the nontrivial eigenvalue of $s$ in $V^*$.  Take $c: S \to \CC$ to be a conjugation invariant function on the reflections.  Following \cite{E1}, we define the rational Cherednik algebra $H_c(W)$ associated to $W$ to be the quotient of
\[
\CC[W] \ltimes T(V \oplus V^*)
\]
by the relations
\[
[x,x'] = 0, \qquad [y,y'] = 0, \qquad [y,x] = (y, x) - \sum_{s \in S} c(s) (y, \alpha_s) (x, \alpha_s^\vee) s
\]
for $x,x' \in V^*$ and $y, y' \in V$.  Let $x_i, y_i$ be dual bases of $V^*$ and $V$, respectively. 

\subsection{Standard modules}

For every irreducible representation $\tau$ of $W$, we may define the standard module with lowest weight $\tau$ as
\[
M_c(\tau) = H_c(\tau) \underset{\CC[W] \ltimes  S V} \otimes \tau.
\]
It has a unique irreducible quotient $L_c(\tau) = M_c(\tau)/J_c(\tau)$, where $J_c(\tau)$ is the sum of all proper submodules in $M_c(\tau)$; it is known that $L_c(\tau)$ are exactly the simple objects in the category $\OO$ of representations of $H_c(W)$.  

\subsection{Weights and characters}
Define the element $\hh = \sum_i x_i y_i + \frac{\dim V}{2} - \sum_{s \in S} \frac{2 c_s}{1 - \lambda_s} s$ in $H_c(W)$.  This element is $W$-invariant and satisfies the properties $[\hh,x] = x$ and $[\hh,y] = -y$.  In \cite{GGOR}, a category $\OO$ of $H_c(W)$-modules was defined to consist of those $H_c(W)$-modules that are the direct sums of finite dimensional generalized eigenspaces of $\hh$ and on which the real part of the spectrum of $\hh$ is bounded below.  Let $\bar{\OO}$ be the category of $H_c(W)$-modules where instead the real part of the spectrum of $\hh$ is bounded above.  

For $M$ in $\OO$ or $\bar{\OO}$, the character $\chi_M(g, t)$ of a $H_c(W)$-module $M$ is defined to be $\Tr|_M(gt^{\hh})$ (as a series in $t$).  For instance, we have that 
\[
\chi_{M_c(\tau)}(g, t) = t^{h(\tau)}\frac{\chi_\tau(g)}{\det_{V^*}(1-gt)},
\]
where $h(\tau) = \frac{\dim V}{2} - \sum_{s \in S} \frac{2 c_s}{1 - \lambda_s} s|_\tau$ denotes the lowest eigenvalue of $\hh$ on $M_c(\tau)$.  

\subsection{Notations}
We specialize to the complex reflection group $W = G_4 = \langle a, b \mid a^3=b^3=1, aba = bab \rangle$ (see \cite{ST}).  Fix $a$ to have non-trivial eigenvalue $e^{\frac{4 \pi i}{3}}$ in the reflection representation.  There are two conjugacy classes $\langle a \rangle$ and $\langle a^2 \rangle$ of complex reflections; the parameters for $H_c(G_4)$ may therefore be taken to be $c_1 = c(a)$ and $c_2 = c(a^2)$.  Notice that $W$ is the finite subgroup of $SU(2)$ corresponding to the affine Dynkin diagram $\widetilde{E}_6$ in the McKay correspondence, so we may characterize its representations as follows. The irreducible representations of $W$ consist of the reflection representation $V$, its dual $V^*$, the embedding of $W$ into $SU(2)$, a three dimensional representation $\CC^3$, the trivial character $\CC$, and two characters $\CC_-$ and $\CC_+$ such that $V \otimes W = \CC^3 \oplus \CC_+$ and $V^* \otimes W = \CC^3 \oplus \CC_-$.  We will use the notations $\CC_0 = \CC$, $\CC_1 = \CC_+$, $\CC_2 = \CC_-$, $V_0 = W$, $V_1 = V$, and $V_2 = V^*$. 

It will be convenient for us to introduce new homogeneous parameters $a_0, a_1, a_2$ given by 
\[
a_0 = 2(\zeta c_1 + \zeta^{-1} c_2), \qquad a_1 = 2(\zeta^{-1}c_1 + \zeta c_2), \qquad \text{ and } a_2 = - 2(c_1 + c_2),
\]
where $\zeta = e^{\frac{2 \pi i}{6}}$. Here, we have $a_0 + a_1 + a_2 = 0$.  We will use $H_a(W)$ or $H_{a_0, a_1, a_2}(W)$ to denote the rational Cherednik algebra with the corresponding values of $(c_1, c_2)$ as parameters.  Similarly, we will write $M_a(\tau)$ and $L_a(\tau)$ for the standard module and its irreducible quotient.

\section{Results}

In this section, we describe the structure of finite dimensional representations of $H_a(W)$.  Away from a periodic family of lines, there are no finite dimensional representations.  For points on exactly one of these lines, there is only one finite dimensional irreducible representation, which is described in Theorem \ref{lines}.  At intersections of the lines, there can be more than one finite dimensional irreducible representation, and these representations are classified in Theorem \ref{points}.

\begin{theorem} \label{lines}
If $H_a(W)$ has a finite dimensional irreducible representation $M$, then the parameters $(a_0, a_1, a_2)$ lie, for some integer $m$, on one of the lines
\begin{itemize}
\item $a_i - a_{i-1} = 3m + 3/2$;

\item $a_i = 3m + 3/2, 3m + 5/2$,
\end{itemize}
where we take indices modulo $3$ and define the first group of lines to have type (1) and the second to have type (2).  At every point on these lines, there exists a finite dimensional representation with character given as follows.

\begin{table}[h!]
\begin{tabular}{|l|l|l|l|l|l|l|l|l} \hline
Type & Line & Character  \\ \hline 
(1)&$a_i - a_{i-1} = 3m + 3/2$ &  $\chi_{M_a(\CC_{-i})} + \chi_{M_a(V_{-i})} - \chi_{M_a(\CC^3)}$ \\ \hline 
(2)&$a_i = 3m + 3/2, 3m + 5/2$ & $\chi_{M_a(\CC_{-i})} + \chi_{M_a(\CC_{- i+2})} - \chi_{M_a(V_{-i+1})}$ \\ \hline 
\end{tabular}
\end{table}

For points on only one of the lines, this representation is the unique indecomposable finite dimensional representation (and is therefore irreducible).
\end{theorem}
Theorem \ref{lines} gives a complete classification of finite dimensional representations away from intersections of the lines.  It then remains to examine the behavior at these intersections.  In Figure \ref{fig-lines}, notice that the lines form a wallpaper pattern in the plane of type $*632$ (see \cite{Con} for an explanation of these patterns). Denote the type of an intersection point by the types of the lines on which it lies.  We then see that there are three types of intersections in the pattern, $(11)$, $(22)$, and $(1122)$.  Theorem \ref{points} gives a classification of the finite dimensional representations at these points.

\begin{figure}[h] 
\caption{Lines on which finite dimensional representations exist. \label{fig-lines}}
\includegraphics[width=5in]{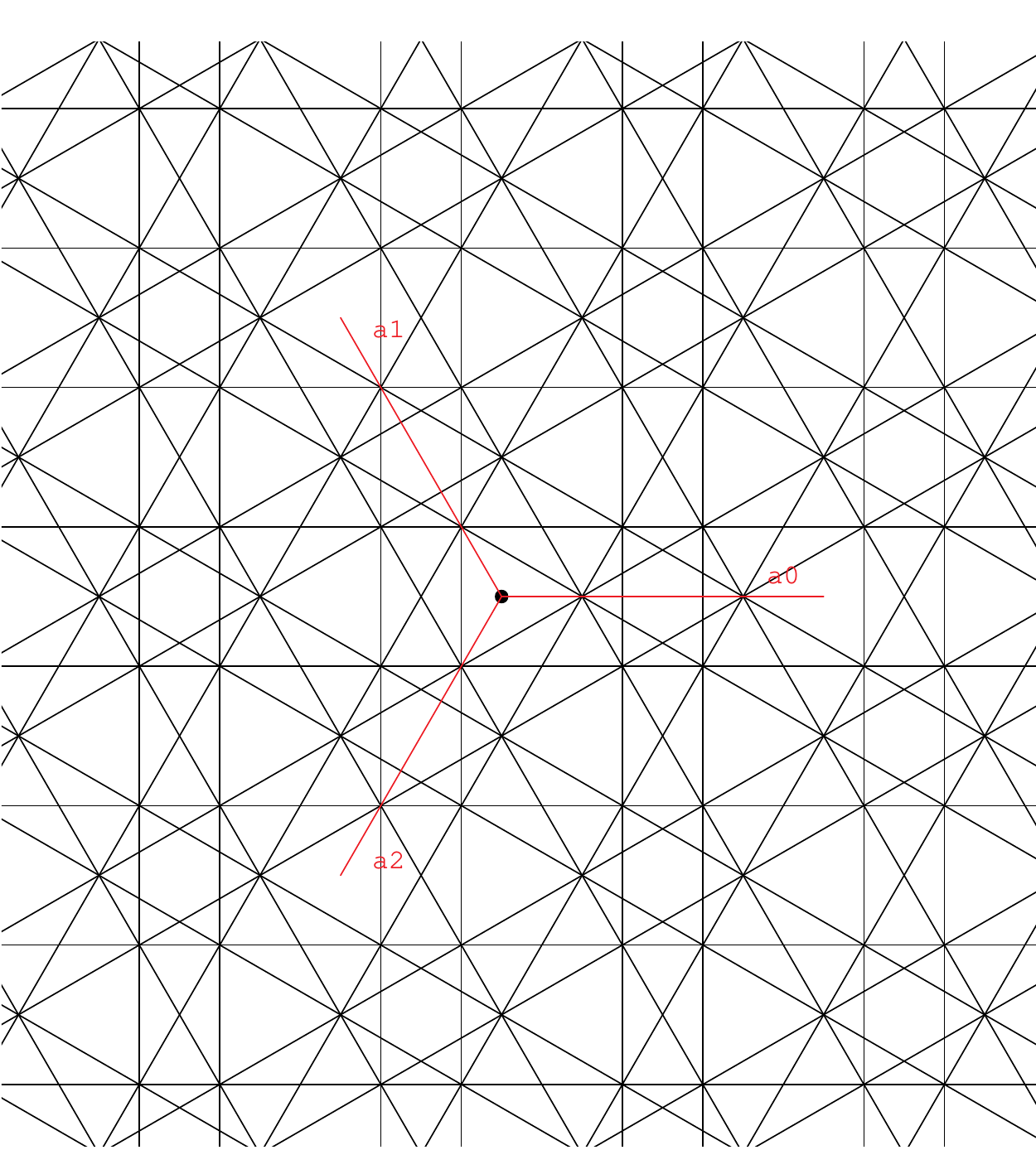}
\end{figure}

\begin{theorem} \label{points}
Points of type $(11)$ take the form 
\[
(a_i, a_{i+1}, a_{i+2}) = (m_i - m_{i+1}, m_i - 2 m_{i+1} + 3/2, -2 m_i - m_{i+1} - 3/2)
\]
for $m_i, m_{i+1} \not \equiv 0 \pmod{3}$.  At these points, there are exactly two finite dimensional irreducible representations, whose characters are given by the following table.

\begin{table}[h!]
\begin{tabular}{|l|l|} \hline
\multicolumn{1}{|c|}{Class of Points} & \multicolumn{1}{|c|}{Characters of Irreducible Representations} \\ \hline
\multirow{2}{*}{$m_i m_{i+1} > 0, m_i (m_i - m_{i+1}) > 0$} & 
$\chi_{M_a(\CC_{-i+2})} +\chi_{M_a(V_{-i+2})} - \chi_{M_a(\CC^3)}$\\
 & $\chi_{M_a(\CC_{-i})} - \chi_{M_a(\CC_{-i+2})} + \chi_{M_a(V_{-i})} - \chi_{M_a(V_{-i+2})}$ \\ \hline
 \multirow{2}{*}{$m_i m_{i+1} > 0, m_i (m_i - m_{i+1}) < 0$} & 
$\chi_{M_a(\CC_{-i})} +\chi_{M_a(V_{-i})} - \chi_{M_a(\CC^3)}$\\
 & $\chi_{M_a(\CC_{-i+2})}-\chi_{M_a(\CC_{-i})}  - \chi_{M_a(V_{-i})} + \chi_{M_a(V_{-i+2})}$ \\ \hline
\multirow{2}{*}{$m_i m_{i+1} < 0$} & $\chi_{M_a(\CC_{-i})} +\chi_{M_a(V_{-i})} - \chi_{M_a(\CC^3)}$  \\
 & $\chi_{M_a(\CC_{-i+2})} +\chi_{M_a(V_{-i+2})} - \chi_{M_a(\CC^3)}$  \\ \hline
\end{tabular}
\end{table}

Points of type $(22)$ take the form 
\[
(a_i, a_{i+1}, a_{i+2}) = (m_i + 1/2, m_{i+1} + 1/2, -m_i -m_{i+1} - 1)
\]
for $m_i, m_{i+1}, m_{i+1} - m_i \not \equiv 0 \pmod{3}$. At these points, there are exactly two finite dimensional irreducible representations, whose characters are given by the following table.

\begin{table}[h!]
\begin{tabular}{|l|l|} \hline
\multicolumn{1}{|c|}{Class of Points} & \multicolumn{1}{|c|}{Characters of Irreducible Representations} \\ \hline
\multirow{2}{*}{$m_i m_{i+1} > 0$} & 
$\chi_{M_a(\CC_{-i})} + \chi_{M_a(\CC_{-i+2})} - \chi_{M_a(V_{-i+1})}$\\
 & $\chi_{M_a(\CC_{-i+1})} + \chi_{M_a(\CC_{-i+2})} - \chi_{M_a(V_{-i})}$ \\ \hline
\multirow{2}{*}{$m_i m_{i+1} < 0$, $m_i(m_i + m_{i+1} + 1) > 0$} & $\chi_{M_a(\CC_{-i})} +\chi_{M_a(\CC_{-i+2})} - \chi_{M_a(V_{-i+1})}$ \\
 &  $\chi_{M_a(\CC_{-i+1})}  - \chi_{M_a(\CC_{-i})}- \chi_{M_a(V_{-i})} + \chi_{M_a(V_{-i+1})}$ \\ \hline
\multirow{2}{*}{$m_i m_{i+1} < 0$, $m_i(m_i+m_{i+1} + 1) < 0$} &  
$\chi_{M_a(\CC_{-i+1})} + \chi_{M_a(\CC_{-i+2})} - \chi_{M_a(V_{-i})}$\\
 & $\chi_{M_a(\CC_{-i})} - \chi_{M_a(\CC_{-i+1})} + \chi_{M_a(V_{-i})} - \chi_{M_a(V_{-i+1})}$ \\ \hline
\end{tabular}
\end{table}

Points of type $(1122)$ take the form 
\[
(a_i, a_{i+1}, a_{i+2}) = (m_i + 1/2, m_{i+1} + 1/2, -m_i - m_{i+1} - 1)
\]
for $m_i, m_{i+1} \not\equiv 0 \pmod{3}$ and $m_i \equiv m_{i+1} \pmod{3}$.  At these points, there are exactly three finite dimensional irreducible representations, whose characters are given in the following table.

\begin{table}[h!]
\begin{tabular}{|l|l|} \hline
\multicolumn{1}{|c|}{Class of Points} & \multicolumn{1}{|c|}{Characters of Irreducible Representations} \\ \hline
\multirow{3}{*}{$m_im_{i+1} > 0$} & 
$\chi_{M_a(\CC_{-i})} + \chi_{M_a(\CC_{-i+2})} - \chi_{M_a(V_{-i+1})}$\\
 & $\chi_{M_a(\CC_{-i+1})} + \chi_{M_a(\CC_{-i+2})} - \chi_{M_a(V_{-i})}$ \\
 & $\chi_{M_a(V_{-i})} - \chi_{M_a(\CC_{-i+2})}  + \chi_{M_a(V_{-i+1})} - \chi_{M_a(\CC^3)}$ \\ \hline
\multirow{3}{*}{$m_i m_{i+1} < 0$, $m_i (m_i + 2m_{i+1} + 3/2) > 0$} & $\chi_{M_a(\CC_{-i})} + \chi_{M_a(\CC_{-i+2})} - \chi_{M_a(V_{-i+1})}$ \\
&$\chi_{M_a(\CC_{-i+1})} + \chi_{M_a(V_{-i+1})} - \chi_{M_a(\CC^3)}$ \\
&$\chi_{M_a(\CC_{-i})} - \chi_{M_a(\CC_{-i+1})} + \chi_{M_a(V_{-i})} - \chi_{M_a(V_{-i+1})}$ \\ \hline
$m_i m_{i+1} < 0$, & $\chi_{M_a(\CC_{-i})} + \chi_{M_a(V_{-i})} - \chi_{M_a(\CC^3)}$ \\
$m_i (m_i + 2m_{i+1} + 3/2) < 0$,&$\chi_{M_a(\CC_{-i+1})} + \chi_{M_a(V_{-i+1})} - \chi_{M_a(\CC^3)}$ \\ 
$m_i(2m_i + m_{i+1} + 3/2) > 0$&$\chi_{M_a(\CC_{-i+2})} - \chi_{M_a(V_{-i})} - \chi_{M_a(V_{-i+1})} + \chi_{M_a(\CC^3)}$ \\ \hline
\multirow{3}{*}{$m_i m_{i+1} < 0$, $m_i (2m_i + m_{i+1} + 3/2) < 0$} &
$\chi_{M_a(\CC_{-i})} + \chi_{M_a(\CC_{-i+2})} - \chi_{M_a(V_{-i+1})}$ \\
& $\chi_{M_a(\CC_{-i})} + \chi_{M_a(V_{-i})} - \chi_{M_a(\CC^3)}$ \\
& $\chi_{M_a(\CC_{-i+1})} - \chi_{M_a(\CC_{-i})}  - \chi_{M_a(V_{-i})} + \chi_{M_a(V_{-i+1})}$ \\ \hline
\end{tabular}
\end{table}
\end{theorem}

Notice that many of the irreducible representations given by Theorem \ref{points} are quotients of the representations constructed along lines in Theorem \ref{lines}; however, at some points, these representations can interact, giving irreducible representations with different lowest weight.

\section{Preliminaries for the Proofs}

In this section, we recall or prove a number of results which will be used in our proofs.

\subsection{Lowest Weights}

Note that the lowest eigenvalue of $\hh$ on $M_c(\tau)$ is given by $h(\tau)$; this means that, for any $W$-subrepresentation $\sigma \subset M_c(\tau)$, we must have that $h(\sigma) - h(\tau) \in \NN$.  Therefore, knowledge of the values of $h(\tau)$ given below will allow us to understand the possible maps between standard modules. 

\begin{lemma} \label{weights}
The lowest weights $h(\tau)$ of each representation are as follows.
\begin{table}[h!]
\begin{tabular}{|l|c|c|c|} \hline	
$\tau$ & $\CC_i$ & $V_i$ & $\CC^3$ \\ \hline
$h(\tau)$ & $4/3(a_{2-i} - a_{-i}) + 1$ & $2/3(a_{-i} - a_{2-i}) + 1$ & $1$\\ \hline
\end{tabular}
\end{table}
\end{lemma}
\begin{proof}
This follows from the computation of $h(\tau) = 1 - \sum_{s \in S} \frac{2 c_s}{1 - \lambda_s} s|_\tau$.
\end{proof}

We will often need to compare complex weights $a, b$ which differ by a real number; for convenience, we will write $a > b$ to mean that $a - b$ is a positive real number.

\subsection{Characters}

We relate the characters of $L_a(\tau)$ and $M_a(\tau)$ by relating the corresponding elements of the Grothendieck group.  The following well-known lemma allows us to represent these using the classes of standard modules only.

\begin{lemma}[\cite{GGOR}]  \label{basis}
Let $K = K(\OO(H_a(W)))$ be the Grothendieck group of representations of $H_a(W)$ in category $\OO$.  Then the classes $[M_a(\sigma)]$ for $\sigma$ an irreducible representation of $W$ form a $\ZZ$-basis for $K$.
\end{lemma}

By considering lowest weights, notice that $[M_a(\sigma)]$ can only appear in the decomposition of $[L_a(\tau)]$ if $h(\tau) > h(\sigma)$.  One of our tools for the main proofs will be the following lemma, which restricts the possibilities for the character of a finite dimensional irreducible $H_a(W)$-module $M$.

\begin{lemma} \label{poletrick}
Let $M$ be a finite dimensional irreducible $H_a(W)$-module with $[M] = \sum_\sigma n_\sigma [M_a(\sigma)]$ for some $n_\sigma$ by Lemma \ref{basis}.  Let $t_{g, 1}$ and $t_{g, 2}$ be the two (not necessarily distinct) values of $t$ for which $\det_{V^*}(1 - gt) = 0$.  Let $A$ be the matrix with columns indexed by irreducible representations $\sigma$ of $W$ and rows by ordered pairs of a conjugacy class $[g]$ in $W$ and an element $i$ in $\{1, 2\}$ such that $A_{(g, i), \sigma} = t_{g, i}^{h(\sigma)} \chi_\sigma(g)$.  Then, the vector $N$ with entries $n_\sigma$ is in the nullspace of $A$.
\end{lemma}
\begin{proof}
Suppose $M$ has lowest weight $\tau$; this means that  $M = L_a(\tau)$ must be finite dimensional.  Note here that $n_{\tau} = 1$ because $L_a(\tau)$ is a lowest weight module with lowest weight $\tau$.  Therefore, we see that
\begin{equation} \label{char}
\chi_{L_a(\tau)}(g,t) = \sum_\sigma n_{\sigma} \chi_{M_a(\sigma)}(g,t) = \sum_{\sigma} \frac{n_{\sigma} t^{h(\sigma)} \chi_\sigma(g)}{\det_{V^*}(1 - gt)} = \frac{1}{\det_{V^*}(1 - gt)} \sum_\sigma n_{\sigma} t^{h(\sigma)} \chi_\sigma(g).
\end{equation}
For fixed $g$, viewed as a function of $t$, $\chi_{L_a(\tau)}(g, t)$ is a linear combination of rational powers of $t$; in particular, it has no poles.  Recalling the definitions of $t_{g,1}$ and $t_{g, 2}$, we find for all $g \in W$ that 
\[
\sum_\sigma n_{\sigma} t_{g, i}^{h(\sigma)} \chi_\sigma(g) = 0 \text{ for } i = 1, 2.
\]
Notice here that $\chi_\sigma(g)$, $t_{g, 1}$, and $t_{g, 2}$ are conjugation invariant functions of $g$, hence we obtain one constraint for each conjugacy class in $W$.  Taking $A$ to be the matrix with entries $A_{(g, i), \sigma} = t_{g, i}^{h(\sigma)} \chi_\sigma(g)$, this is exactly the desired conclusion.
\end{proof}

Let us now compute the matrix $A$ of Lemma \ref{poletrick} explicitly in our case.  Set $y = \zeta^{2a_0}$ and $w = \zeta^{2a_1}$. Using Magma to compute $t_{g, 1}$, $t_{g, 2}$, $h(\sigma)$ and $\chi_\sigma(g)$ for each $\sigma$ and $g$, the matrix $A$ is given by
\begin{equation} \label{matrix}
A = 
\left[\begin{matrix}
1& 1& 1& 2& 2& 2& 3\\ 
y^{-4}w^{-2}& y^{2}w^{4}& y^{2}w^{-2}& -2y^{2}w& -2y^{-1}w^{-2}& -2y^{-1}w& 3\\ 
1& \zeta^{4}& \zeta^{2}& \zeta^{3}& \zeta& \zeta^{5}& 0 \\
y^{\frac{-16 }{ 3 }}w^{\frac{ -8 }{ 3 }}& \zeta^{2}y^{\frac{ 8 }{ 3 }}w^{\frac{ 16 }{ 3 }}& \zeta^{4}y^{\frac{ 8 }{ 3 }}w^{\frac{ -8 }{ 3 }}& \zeta^{3}y^{\frac{ 8 }{ 3 }}w^{\frac{ 4 }{ 3 }}& \zeta^{5}y^{\frac{ -4 }{ 3 }}w^{\frac{ -8 }{ 3 }}& \zeta y^{\frac{ -4 }{ 3}}w^{\frac{ 4 }{ 3 }}& 0 \\
y^{-2}w^{-1}& y w^{2}& yw^{-1}& 0& 0& 0& -1\\ 
y^{\frac{ -20 }{ 3 }}w^{\frac{ -10 }{ 3 }}& \zeta^{4}y^{\frac{ 10 }{ 3 }}w^{\frac{ 20 }{ 3 }}& \zeta^{2}y^{\frac{ 10 }{ 3 }}w^{\frac{ -10 }{ 3}}& 
y^{\frac{ 10 }{ 3 }}w^{\frac{ 5 }{ 3 }}& \zeta^{4}y^{\frac{ -5 }{ 3 }}w^{\frac{ -10 }{ 3 }}& \zeta^{2}y^{\frac{ -5 }{ 3 }}w^{\frac{ 5 }{ 3 }}& 0 \\
y^{-4}w^{-2}& \zeta^{2}y^{2}w^{4}& \zeta^{4}y^{2}w^{-2}& y^{2}w& \zeta^{2}y^{-1}w^{-2}& \zeta^{4}y^{-1}w& 0 \\
1& 1& 1& 2& 2& 2& 3\\ 
y^{-4}w^{-2}& y^{2}w^{4}& y^{2}w^{-2}& -2y^{2}w& -2y^{-1}w^{-2}& -2y^{-1}w& 3\\ 
y^{\frac{ -8 }{ 3 }}w^{\frac{ -4 }{ 3 }}& \zeta^{4}y^{\frac{ 4 }{ 3 }}w^{\frac{ 8 }{ 3 }}& \zeta^{2}y^{\frac{ 4 }{ 3 }}w^{\frac{ -4 }{ 3}}& 
\zeta^{3}y^{\frac{ 4 }{ 3 }}w^{\frac{ 2 }{ 3 }}& \zeta y^{\frac{ -2 }{ 3 }}w^{\frac{ -4 }{ 3 }}& \zeta^{5}y^{\frac{ -2 }{ 3 }}w^{\frac{ 2 }{ 3 }}& 0 \\
1& \zeta^{2}& \zeta^{4}& \zeta^{3}& \zeta^{5}& \zeta& 0 \\
y^{-6}w^{-3}& y^{3}w^{6}& y^{3}w^{-3}& 0& 0& 0& -1\\ 
y^{-4}w^{-2}& \zeta^{4}y^{2}w^{4}& \zeta^{2}y^{2}w^{-2}& y^{2}w& \zeta^{4}y^{-1}w^{-2}& \zeta^{2}y^{-1}w& 0 \\
y^{\frac{ -4 }{ 3 }}w^{\frac{ -2 }{ 3 }}& \zeta^{2}y^{\frac{ 2 }{ 3 }}w^{\frac{ 4 }{ 3 }}& \zeta^{4}y^{\frac{ 2 }{ 3 }}w^{\frac{ -2 }{ 3 }}& y^{\frac{ 2 }{ 3 }}w^{\frac{ 1 }{ 3 }}& \zeta^{2}y^{\frac{ -1 }{ 3 }}w^{\frac{ -2 }{ 3 }}& \zeta^{4}y^{\frac{ -1 }{ 3 }}w^{\frac{ 1 }{ 3 }}& 0 \\
\end{matrix}\right].
\end{equation}

\subsection{Operations on representations}

In Lemmas \ref{flip} and \ref{rotate}, we introduce two operations which may be performed on $H_a(W)$-modules.  These operations will allow us to use the classification of $H_a(W)$-modules for some value of $a$ to classify $H_a(W)$ modules for a different value of $a$.

\begin{lemma} \label{flip}
There is a duality of categories $\delta: \OO(H_{a_0, a_1, a_2}(W)) \to \bar{\OO}(H_{a_1, a_0, a_2}(W))$ sending a $H_{a_0, a_1, a_2}(W)$-module $M$ to a $H_{a_1, a_0, a_2}(W)$-module $\delta(M)$ isomorphic to $M^*$ as a $W$-module.
\end{lemma}
\begin{proof}
For any $H_{a_0, a_1, a_2}(W)$-module $M$, let $\delta(M) = M^*$.  Note that $M^*$ is naturally a $H_{a_0, a_1, a_2}(W)^\text{op}$ module, so it suffices to give a map $\phi: H_{a_1, a_0, a_2}(W) \to H_{a_0, a_1, a_2}(W)^\text{op}$.  We claim that $\phi(g) = g^{-1}$, $\phi(x) = x$ and $\phi(y) = -y$ for $g \in W$, $x \in V^*$, and $y \in V$ induces a valid map.  Indeed, all necessary relations are evident except for $[\phi(y), \phi(x)]^\text{op} = \phi([y,x])$.  To check this relation, let $(c_1, c_2)$ be the parameters corresponding to $(a_0, a_1, a_2)$.  Notice that $(c_1, c_2) = \left(\frac{1}{2} \frac{\zeta}{\zeta^4 - 1} (a_0 \zeta^2 - a_1), \frac{1}{2} \frac{\zeta}{\zeta^4 - 1} (a_1 \zeta^2 - a_0)\right)$, so by symmetry $(c_2, c_1)$ are the parameters corresponding to $(a_1, a_0, a_2)$.  Now, notice that 
\[
[\phi(y), \phi(x)]^\text{op} = [-y, x]^\text{op} (y, x) + \sum_{s \in S} c(s)(y, \alpha_s)(x, \alpha_s^\vee) s 
	= (y, x) - \sum_{s \in S} c(s^{-1}) (-y, \alpha_s)(x, \alpha_s^\vee) s^{-1} = \phi([y, x]),
\]
where the last equality holds because $c(s) = c_1$ implies $c(s^{-1}) = c_2$.  The resulting map $\delta = M \mapsto M^*$ is evidently natural and self-dual, completing the proof.
\end{proof}

\begin{lemma} \label{rotate}
There is an equivalence of categories $\omega: \OO(H_{a_0, a_1, a_2}(W)) \to \OO(H_{a_2, a_0, a_1}(W))$ sending a $H_{a_0, a_1, a_2}(W)$-module $M$ to a $H_{a_2, a_0, a_1}(W)$ module $\omega(M)$ isomorphic to the tensor product of $M$ and $\CC_2$ as a $W$-module.
\end{lemma}
\begin{proof}
Let $\chi_{\CC_i}$ be the character of $\CC_i$ as a $W$-representation.  We first claim that there is an isomorphism $\phi: H_{a_2, a_0, a_1}(W) \to H_{a_0, a_1, a_2}(W)$ given by $\phi(g) = \chi_{\CC_i}(g) \cdot g$, $\phi(x) = x$, and $\phi(y) = y$ for $g \in W, x \in V^*, y \in V$.   Note that all relations hold immediately except for $[\phi(y),\phi(x)] = \phi([y,x])$.  For this, let $(c_1, c_2)$ be the parameters corresponding to $(a_2, a_0, a_1)$; as in the proof of Lemma \ref{flip}, we have that $(c_1, c_2) = \left(\frac{1}{2} \frac{\zeta}{\zeta^4 - 1} (a_2 \zeta^2 - a_0), \frac{1}{2} \frac{\zeta}{\zeta^4 - 1} (a_0 \zeta^2 - a_2)\right)$ and $(a_2, a_0) = (2 \zeta c_1 + 2 \zeta^{-1} c_2, 2 \zeta^{-1} c_1 + 2 \zeta c_2)$.  Therefore, the parameters corresponding to $(a_0, a_1, a_2)$ are 
\[
(c_1', c_2') = \left(\frac{1}{2} \frac{\zeta}{\zeta^4 - 1}(-(a_2 + a_0) \zeta^2 - a_2), \frac{1}{2} \frac{\zeta}{\zeta^4 - 1}(a_2 \zeta^2 + a_2 + a_0)\right) = (\zeta^2 c_1, \zeta^4 c_2).
\]
Now, we may check that 
\[
\phi([y,x])= (y,x) - \sum_{s \in S} c(s)(y,\alpha_s)(x, \alpha_s^\vee) \chi_{\CC_2}(s) \cdot s = (y, x) - \sum_{s \in S} c'(s) (y, \alpha_s)(x, \alpha_s^\vee) s  =[\phi(y), \phi(x)],
\]
where we use the fact that $\chi_{\CC_2}$ takes values of $\zeta^2$ and $\zeta^4$ on the two conjugacy classes of reflections. Thus $\phi$ is well-defined; iterating it three times then shows that it is an isomorphism.

Hence, the category $\OO$ of each isomorphic algebra $H_{a_2, a_0, a_1}(W)$ and $H_{a_0, a_1, a_2}(W)$ is equivalent.  In particular, this equivalence sends a $H_{a_0, a_1, a_2}(W)$ module $M$ to a $H_{a_2, a_0, a_1}(W)$-module $\omega(M)$ isomorphic to $M$ as a vector space with action induced by $\phi$.  The $W$-action on $\omega(M)$ is given by $g \cdot m' = \chi_{\CC_2}(g) \cdot (g \cdot m')$; this is exactly the $W$ action on $M \otimes \CC_2$, so $\omega(M)$ is isomorphic to $M \otimes \CC_2$ as a $W$-module, completing the proof.
\end{proof}

Lemmas \ref{flip} and \ref{rotate} give functors between category $\OO(H_a(W))$ for different values of $a$.  Notice that there is a $S_3$-action on the plane compatible with these functors. The functor $M \mapsto M^*$ corresponds to the flip $(a_0, a_1, a_2) \mapsto (a_1, a_0, a_2)$, and $M \mapsto M \otimes \CC_2$ corresponds to the rotation $(a_0, a_1, a_2) \mapsto (a_2, a_0, a_1)$.

\subsection{Hecke algebras and the $\KZ$ functor}

We briefly review the construction of the $\KZ$ functor, which associates to a $H_a(W)$-module $M$ in category $\OO$ a module $\KZ(M)$ over the Hecke algebra $\HH(\{q_{H, j}\})$ of $W$.  This construction is given in more detail in \cite{BMR} and \cite{GGOR}.  We then use it to prove a lemma about mappings between standard modules.

\subsubsection{Definition of the Hecke algebra}
We first recall the definition of the Hecke algebra associated to a complex reflection group $W$.  For $s \in S$ a reflection in $W$, let $H_s \subset V$ be the hyperplane it preserves.  Now, let $\AAA = \{H_s \mid s \in S\}$ and $V_\text{reg} = V \setminus \bigcup_{s \in S} H_s$; notice that there is a $W$-action on $V_\text{reg}$.  Then, the Artin braid group $B_W$ associated to $W$ is given by $\pi_1(V_\text{reg}/W, x_0)$ for some $x_0 \in V_\text{reg}$.  

Now, let $\{q_{H,j}\}_{H \in \AAA/W, 0 \leq j < e_H}$, where $e_H$ is the order of some $s \in S$ preserving $H$, be a set of parameters such that $q_{H,j} = q_{H',j}$ if $H$ and $H'$ are in the same $W$-orbit.  Then, the Hecke algebra $\HH(\{q_{H,j}\})$ of $W$ with parameters $\{q_{H,j}\}$ is defined to be the quotient of $\CC[B_W]$ by the relations 
\[
(T_s - q_{H_s, 0})(T_s - q_{H_s, 1}) \cdots (T_s - q_{H_s, e_s - 1}) = 0.
\]
By a specialization of $\HH(\{q_{H, j}\})$ we mean a choice for the parameters $\{q_{H, j}\}$.  Notice that for the specialization $q_{H,j} = \det_V(s_H)^{-j}$ for $s_H$ a distinguished reflection about $H$, we have that $\HH(\{\det_V(s_H)^{-j}\}) = \CC[W]$.

In the case $W = G_4$, we notice that $W$ acts transitively on the set of reflecting hyperplanes $H_s$ in $V$, meaning that there is only one value of $q_{H, j} = q_j$ for each $j$.  Further, there are exactly two classes of reflections, with representative elements $s$ and $t$ of order $3$, and it is known (see \cite{BMR}) that $B_W$ has presentation
\[
B_W = \Big\langle T_s, T_t \mid T_s T_t T_s = T_t T_s T_t\Big\rangle.
\]
Therefore, the Hecke algebra has explicit presentation 
\[
\HH(\{q_j\}) = \Big\langle T_s, T_t \mid T_sT_tT_s = T_tT_sT_t, (T_r - q_0)(T_r - q_1)(T_r-q_2) = 0 \text{ for $r = s, t$}\Big\rangle.
\]

\subsubsection{The $\KZ$ functor}

We now describe the $\KZ$ functor.  For any $H_a(W)$ module $M$ that is finitely generated over $\CC[V]$, we may define its localization 
\[
M_\text{loc} = M \underset{\CC[V]} \otimes \CC[V_\text{reg}],
\]
which is naturally endowed with an action of $H_a(W)_\text{loc}$; here we recall that $V_\text{reg} = V \setminus \bigcup_{s \in S} H_s$.  Notice that there is an isomorphism $H_a(W)_\text{loc} \to W \ltimes \DD(V_\text{reg})$ which endows $M_\text{loc}$ with a structure of $W$-equivariant $\DD(V_\text{reg})$-module.  Because $M$ is finitely generated over $\CC[V]$, we may regard $M_\text{loc}$ as a vector bundle over $V_\text{reg}$ equipped with a flat connection $\nabla$ via the $\DD(V_\text{reg})$-module structure.  Taking horizontal sections of this connection gives a monodromy representation
\[
\rho_M: \pi_1(V_\text{reg}/W, x_0) \to GL(\CC^N)
\]
with $N$ the rank of $M_\text{loc}$.  This defines the Knizhnik-Zamolodchikov functor $\KZ(M) = \rho_M$ between $H_a(W)$-modules finitely generated over $\CC[V]$ and finite dimensional $B_W$-modules.  

Consider now the specialization $q_{H, j}^* = \det(s_H)^{-j} e^{2 \pi i k_{H,j}}$, where $k_{H, j}$ for $H \in \AAA/W$ is defined by $k_{H, 0} = 0$ and $c_{s_H} = -\frac{1}{2}\sum_{j = 0}^{e_H - 1} \det(s_H)^j (k_{H, j + 1} - k_{H, j})$ for all $H, j$.  In this case, it is known (see Theorem 4.12 in \cite{BMR}) that $\rho_M$ factors through $\HH(\{q_{H, j}^*\})$.  We will also denote the resulting representation of $\HH(\{q_{H, j}^*\})$ by $\KZ(M)$.  We primarily consider the effect of $\KZ$ on the standard modules $M_a(W)$.  By Tits' Deformation Theorem, we see that, at generic points, the simple modules over $\HH(\{q_{H, j}^*\})$ are in bijection with the irreducible representations of $W$; let $\widehat{\tau}$ denote the representation of $\HH(\{q_{H, j}^*\})$ corresponding to the representation $\tau$ of $W$.  It is well known (see \cite{GGOR} Section 5.3) that if $\widehat{\tau}$ is irreducible, then it is equal to $\KZ(M_a(\tau))$. We will use this fact in conjunction with the following two results.

\begin{theorem}[Lemma 2.10 in \cite{BEG2}] \label{KZinj}
Let $N$ be a $H_a(W)$-module torsion-free over $\CC[V]$.  Then, for any $M \in \OO$, the natural map
\[
\Hom_{H_a(W)}(M, N) \to \Hom_{\HH(\{q_{H,j}^*\})}(\KZ(M), \KZ(N))
\]
is injective.
\end{theorem}

\begin{corr} \label{dimineq}
For irreducible representations $\sigma$ and $\tau$ of $W$ such that $\widehat{\sigma}$ and $\widehat{\tau}$ are irreducible, we have 
\[
\dim \Hom_{H_a(W)}(M_a(\sigma), M_a(\tau)) \leq \dim \Hom_{\HH(\{q_{H, j}^*\})}(\widehat{\sigma}, \widehat{\tau}).
\]
\end{corr}
\begin{proof}
This follows from Theorem \ref{KZinj} and the fact that $\KZ(M_a(\sigma)) = \widehat{\sigma}$ if $\widehat{\sigma}$ is irreducible.
\end{proof}

Note that $\dim \Hom_{\HH(\{q_{H, j}^*\})}(\widehat{\sigma}, \widehat{\tau}) = 1$ if $\widehat{\sigma}$ and $\widehat{\tau}$ are isomorphic and $\dim \Hom_{\HH(\{q_{H, j}^*\})}(\widehat{\sigma}, \widehat{\tau}) = 0$ otherwise.

\subsubsection{Mappings between standard modules}

Specialize now to the case $W = G_4$. In this case, the parameters are $(k_1, k_2) = \left(a_2/3, -a_0/3\right)$.  Now, $W$ acts transitively on $\AAA$ and $e_H = 3$ for all $H$, so we may take the Hecke algebra parameters to be $q_0, q_1, q_2$ with $q_0 = 1$, $q_1 = e^{\frac{2 \pi i}{3}(1 + a_2)}$, and $q_2 = e^{-\frac{2 \pi i}{3}(1 + a_0)}$.  Fix a distinguished reflection $s = a$ in $W$ with $\det(s) = e^{\frac{4 \pi i}{3}}$.  Computing using CHEVIE 4 in GAP 3.4.4, we may find the matrices of the generators $a$ and $b$ in $\widehat{\tau}$.  These representations are listed in Table \ref{HeckeReps}.

\begin{center}
\begin{table}[h!]
\caption{Matrices for the representations of $\HH(\{q_{H, j}\})$ \label{HeckeReps}}
\begin{tabular}{|l|c|c|c|c|c|c|c|} \hline 
$\tau$ &  $\CC_0$ & $\CC_1$ & $\CC_2$ & $V_0$ & $V_1$ & $V_2$ & $\CC^3$\\ \hline
$\rho_{\widehat{\tau}}(a)$ & $1$ & $q_1$ & $q_2$ & $\left(\begin{matrix} q_1 & 0 \\ -q_1 & q_2 \end{matrix}\right)$ & $\left(\begin{matrix} 1 & 0 \\ -1 & q_2 \end{matrix}\right)$ & $\left(\begin{matrix} 1 & 0 \\ -1 & q_1 \end{matrix}\right)$ & $\left(\begin{matrix} q_2 & 0 & 0 \\ q_2 + q_1^2 & q_1 & 0 \\ q_1 & 1 & 1 \end{matrix}\right)$ \\ 
$\rho_{\widehat{\tau}}(b)$ & $1$ & $q_1$ & $q_2$ & $\left(\begin{matrix} q_2 & q_2 \\ 0 & q_1 \end{matrix}\right)$ & $\left(\begin{matrix} q_2 & q_2 \\ 0 & 1 \end{matrix}\right)$  & $\left(\begin{matrix} q_1 & q_1 \\ 0 & 1 \end{matrix}\right)$ & $\left(\begin{matrix} 1 & -1 & q_1 \\ 0 & q_1 & -q_2 - q_1^2 \\ 0 & 0 & q_2 \end{matrix}\right)$ \\ \hline
\end{tabular}
\end{table}
\end{center}

We now apply this machinery to determine when mappings between certain standard modules exist. 

\begin{lemma} \label{mappings}
For $a = (a_0, a_1, a_2)$, if we have
\[
\dim \Hom_{H_a(W)}(M_a(\CC_0), M_a(\CC_1)) > 0 \text{ or } \dim \Hom_{H_a(W)}(M_a(\CC_1), M_a(\CC_0)) > 0,
\]
then $a_2 = 3m + 2$ for some $m \in \ZZ$.
\end{lemma}
\begin{proof}
By Table \ref{HeckeReps}, we see that $\widehat{\CC}_0$ and $\widehat{\CC}_1$ are always irreducible, so applying Corollary \ref{dimineq} we have
\[
\dim \Hom_{\HH(\{q_{H, j}^*\})}(\widehat{\CC}_0, \widehat{\CC}_1) > 0 \text{ or } \dim \Hom_{\HH(\{q_{H, j}^*\})}(\widehat{\CC}_1, \widehat{\CC}_0) > 0.
\]
Hence, there is a map between $\widehat{\CC}_0$ and $\widehat{\CC}_1$, so we have $q_1 = e^{\frac{2 \pi i}{3}(1 + a_2)} = 1$, implying that $a_2 = 3m + 2$ for some $m \in \ZZ$.
\end{proof}

\begin{lemma} \label{mappings2}
For $a = (a_0, a_1, a_2)$, we have
\begin{itemize}
\item[(i)] if
\[
\dim \Hom_{H_a(W)}(M_a(V_0), M_a(\CC_1)) > 0 \text{ or } \dim \Hom_{H_a(W)}(M_a(\CC_1), M_a(V_0)) > 0,
\]
then $a_1 = 3m + 3/2$ or $a_1 = 3m + 5/2$ for some $m \in \ZZ$;

\item[(ii)] if
\[
\dim \Hom_{H_a(W)}(M_a(V_1), M_a(\CC_0)) > 0 \text{ or } \dim \Hom_{H_a(W)}(M_a(\CC_0), M_a(V_1)) > 0,
\]
then $a_0 = 3m + 3/2$ or $a_1 = 3m + 5/2$ for some $m \in \ZZ$.
\end{itemize}
\end{lemma}
\begin{proof}
For (i), by Table \ref{HeckeReps}, we see that $\widehat{V}_0$ is irreducible unless $a_1 = 3m + 3/2, 5/2$ for some $m \in \ZZ$.  But if $\widehat{V}_0$ is irreducible, then $\dim \Hom_{\HH(\{q_{H, j}^*\})}(\widehat{V}_0, \widehat{\CC}_0) = \dim \Hom_{\HH(\{q_{H, j}^*\})}(\widehat{\CC}_0, \widehat{V}_0) = 0$, so by Corollary \ref{dimineq}, we find that $\dim \Hom_{H_a(W)}(M_a(V_0), M_a(\CC_1)) = \dim \Hom_{H_a(W)}(M_a(\CC_1), M_a(V_0)) = 0$, a contradiction. So we must have $a_1 = 3m + 3/2, 5/2$ for some $m \in \ZZ$.  The proof of (ii) is exactly analogous.
\end{proof}

\section{Proofs}

In this section, we give proofs of the main theorems.  Our general approach will be as follows.  We first show that if $M$ is a finite dimensional irreducible representation of $H_a(W)$, then $a$ lies on a specified set of lines and the character of $M$ takes a specified form away from intersections of the lines.  We then use this characterization to construct representations away from intersection points.  Finally, we use the constructed representations to classify finite dimensional irreducible representations at intersections of the lines.

\subsection{Away from intersection points}

The structure of finite dimensional representations away from intersection points is given by Theorem \ref{lines}.  We divide its proof into Propositions \ref{unique} and \ref{exist}. 

\begin{prop} \label{unique}
If $H_a(W)$ has a finite dimensional irreducible representation $M$, then the parameters $(a_0, a_1, a_2)$ lie, for some integer $m$, on one of the lines
\begin{itemize}
\item $a_i - a_{i-1} = 3m + 3/2$;

\item $a_i = 3m + 3/2, 3m + 5/2$,
\end{itemize}
where indices are modulo $3$.  Further, for a point on only one of the lines, $M$ has the character listed below.

\begin{table}[h!]
\begin{tabular}{|l|l|l|l|l|l|l|l|l} \hline
Type & Line & Character ($\chi(g, t)$) \\ \hline 
(1)&$a_i - a_{i-1} = 3m + 3/2$ &  $\chi_{M_a(\CC_{-i})}(g,t) + \chi_{M_a(V_{-i})}(g,t) - \chi_{M_a(\CC^3)}(g,t)$ \\ \hline
(2)&$a_i = 3m + 3/2, 3m + 5/2$ & $\chi_{M_a(\CC_{-i})}(g, t) + \chi_{M_a(\CC_{-i+2})}(g, t) - \chi_{M_a(V_{-i+1})}(g, t)$ \\ \hline \end{tabular}
\end{table}
\end{prop}

\begin{prop} \label{exist}
If the parameters $(a_0, a_1, a_2)$ lie, for some integer $m$, on one of the lines $a_i - a_{i-1} = 3m + 3/2$, there exists a finite dimensional $H_a(W)$ representation with character $\chi_{M_a(\CC_{-i})}(g,t) + \chi_{M_a(V_{-i})}(g,t) - \chi_{M_a(\CC^3)}(g,t)$.  If they lie on one of the lines $a_i = 3m + 3/2, 3m + 5/2$, there exists a finite dimensional $H_a(W)$ representation with character $\chi_{M_a(\CC_{-i})}(g, t) + \chi_{M_a(\CC_{-i+2})}(g, t) - \chi_{M_a(V_{-i+1})}(g, t)$.
\end{prop}

Before giving the proofs of Propositions \ref{unique} and \ref{exist}, we first discuss how they imply Theorem \ref{lines}. Proposition \ref{unique} shows that, if an irreducible finite dimensional representation $M$ exists, then the parameters must lie on one of the desired lines.  By Proposition \ref{exist}, the desired representations exist on these lines; away from intersections of the lines, Proposition \ref{unique} implies that it is unique, giving the result of Theorem \ref{lines}.  We now give proofs for Propositions \ref{unique} and \ref{exist}.

\begin{proof}[Proof of Proposition \ref{unique}]
We first show that the parameters lie on one of the claimed lines.  Suppose $M$ has lowest weight $\tau$ and hence that $M = L_a(\tau)$ is finite dimensional.  By Lemma \ref{basis}, we may write $[M] = [L_a(\tau)] = \sum_\sigma n_\sigma [M_a(\sigma)]$ for some integers $n_\sigma$.  Let $N = (n_{\CC_0}, n_{\CC_1}, n_{\CC_2}, n_{V_0}, n_{V_1}, n_{V_2}, n_{\CC^3})$.  Then, adopting the notations of Lemma \ref{poletrick}, there is a matrix $A$ given in Equation (\ref{matrix}) with $A_{(g,i), \sigma} = t_{g,i}^{h(\sigma)} \chi_\sigma(g)$ such that $N$ is in the nullspace of $A$.  As in Lemma \ref{poletrick}, write $n_1 = n_{\CC_0}$, $n_2 = n_{\CC_1}$, $n_3 = n_{\CC_2}$, $n_4 = n_{V_0}$, $n_5 = n_{V_1}$, $n_6 = n_{V_2}$, and $n_7 = n_{\CC^3}$.  From row 1 and the real and imaginary parts of row 3 in $A$, we find that
\begin{align*}
n_1 + n_2 + n_3 + 2 n_4 + 2 n_5 + 2 n_6 + 3 n_7 &= 0 \\
n_1 - \frac{1}{2} n_2 - \frac{1}{2} n_3 - n_4 + \frac{1}{2} n_5 + \frac{1}{2} n_6 &= 0 \\
- n_2 + n_3 + n_5 - n_6 &= 0.
\end{align*}
Solving, we have that 
\begin{align*}
n_4 = \frac{1}{2}[n_1 - n_2 - n_3 - n_7],
n_5 = \frac{1}{2}[-n_1 + n_2 - n_3 - n_7], 
n_6 = \frac{1}{2}[-n_1 - n_2 + n_3 - n_7].
\end{align*}
Substituting this into our previous equality, we see that 
\[
A X
\left[\begin{matrix} n_1 \\ n_2 \\ n_3 \\ n_7 \end{matrix}\right] = 0 \qquad \text{ for } \qquad X = \left[\begin{matrix}
1 & 0 & 0 & 0 \\
0 & 1 & 0 & 0 \\
0 & 0 & 1 & 0 \\
\frac{1}{2} & -\frac{1}{2} & -\frac{1}{2} & -\frac{1}{2} \\
-\frac{1}{2} & \frac{1}{2} & -\frac{1}{2} & -\frac{1}{2} \\
-\frac{1}{2} & -\frac{1}{2} & \frac{1}{2} & -\frac{1}{2} \\
0 & 0 & 0 & 1
\end{matrix}\right].
\]
This means that the nullspace of the $14 \times 4$ matrix $B = AX$ contains the non-zero integral point $(n_1, n_2, n_3, n_7)^T$; in particular it is non-empty, meaning that all $4 \times 4$ minors of $B$ must be zero.  For a $4$ element subset $J$ of the rows of $B$, let $B_J$ be the $4 \times 4$ minor of $N$ corresponding to the rows in $J$.  Take $I = \langle B_J \rangle$ to be the ideal in $\CC[y,w]$ generated by the numerators of all $4 \times 4$ minors $B_J$.  Computing in Magma, we find that $I$ is a principal ideal with single generator that factors as
\[
I = \Big\langle (w - \zeta^3)(w - \zeta^5)(y - \zeta^3)(y - \zeta^5)(yw - \zeta^3)(y w - \zeta)(y - \zeta^3 w)(y w^2 - \zeta^3)(y^2 w - \zeta^3)\Big\rangle.
\]
Since $B_J = 0$ for all $J$, the ideal $I$ must be trivial, which means that at least one of the factors must vanish.  Recalling that $y = \zeta^{2a_0}$ and $w = \zeta^{2a_1}$, this means that, for some $m \in \ZZ$, at least one of the conditions
\begin{align*}
2 a_0 &= 6m + 3, 6m + 5 \\
2 a_1 &= 6m + 3, 6m + 5 \\
2 a_2 &= 6m + 3, 6m + 5 \\
2(a_0 - a_2) &= 6m + 3\\
2(a_1 - a_0) &= 6m + 3\\
2(a_2 - a_1) &= 6m + 3
\end{align*}
must hold. These are exactly the equations of the desired lines, establishing the first claim.

We now claim that if the parameters are on only one of these lines, then $M = L_a(\tau)$ must have the claimed character.  Each linear factor $f$ in the generator for $I$ defines by substitution a specialization $A_f$ of $A$; let $B_f$ be the corresponding specialization of $B$.  Using Magma, we may compute a basis for the corresponding nullspace of $A_f$ as a matrix with entries either in $\CC(y)$ or in $\CC(w)$.  In each case, we find that the nullspace is one-dimensional with basis vector given by the following table.
\begin{center}
\begin{table}[h!]
\caption{Generators for the nullspace of $A_f$\label{Nsigma}}
\begin{tabular}{|l|c|c|c|c|c|c|c|} \hline 
 & $\CC_0$ & $\CC_1$ & $\CC_2$ & $V_0$ & $V_1$ & $V_2$ & $\CC^3$ \\ \hline
$a_0 = 3m + 3/2, 3m + 5/2$ & $1$ & $0$ & $1$ & $0$ & $-1$ & $0$ & $0$ \\ \hline
$a_1 = 3m + 3/2, 3m + 5/2$ & $0$ & $1$ & $1$ & $-1$ & $0$ & $0$ & $0$ \\ \hline
$a_2 = 3m + 3/2, 3m + 5/2$ & $1$ & $1$ & $0$ & $0$ & $0$ & $-1$ & $0$ \\ \hline
$a_0 - a_2 = 3m + 3/2$ & $1$ & $0$ & $0$ & $1$ & $0$ & $0$ & $-1$\\ \hline
$a_1 - a_0 = 3m + 3/2$ & $0$ & $0$ & $1$ & $0$ & $0$ & $1$ & $-1$ \\ \hline
$a_2 - a_1 = 3m + 3/2$ & $0$ & $1$ & $0$ & $0$ & $1$ & $0$ & $-1$\\ \hline
\end{tabular}
\end{table}
\end{center}
These nullspaces were computed in either $\CC(y)$ or $\CC(w)$, so for generic points along the lines the values of $N$ are given by the table .  For non-generic points, take $I_f$ to be the ideal in $\CC[y,w]$ generated by the numerators of all $3 \times 3$ minors in $B_f$.  For particular values of $y$ and $w$, if the evaluation map $\CC[y, w] \to \CC$ does not kill all of $I_f$, then the nullspace of $B_f$ and hence $A_f$ is one-dimensional, thus spanned by the vectors above.  Computing in Magma, we find that each $I_f$ is principal with single generator given in the table below.

\begin{table}[h!]
\[
\begin{array}{|l|l|} \hline
\text{Lines} & \text{Generator for $I_f$} \\ \hline
a_0 = 3m + 3/2 & (w - 1)^2(w - \zeta^3)^2(w - \zeta^5)(w - \zeta^4) \\ \hline
a_0 = 3m + 5/2 & (w- \zeta^3)(w - \zeta^2)^2(w - \zeta^5)^2(w - \zeta^4) \\\hline
a_1 = 3m + 3/2 & (y - 1)^2(y - \zeta^3)^2(y - \zeta^5)(y - \zeta^4) \\ \hline
a_1 = 3m + 5/2 & (y- \zeta^3)(y - \zeta^2)^2(y - \zeta^5)^2(y - \zeta^4) \\\hline
a_2 = 3m + 3/2 & (w - 1)^2 (w - \zeta^3)^2 (w - \zeta^5) (w - \zeta^4) \\ \hline
a_2 = 3m + 5/2 & (w - \zeta^3)(w - \zeta^2)^2(w - \zeta^5)^2(w - \zeta^4)\\ \hline
a_0 - a_2 = 3m + 3/2 & (y^2-1)^2(y^2-\zeta^2)(y^2-\zeta^4)^2 \\ \hline
a_1 - a_0 = 3m + 3/2 & (w^2-1)^2(w^2-\zeta^2)(w^2-\zeta^4)^2 \\ \hline
a_2 - a_1 = 3m + 3/2 & (w^2-1)^2(w^2-\zeta^2)(w^2-\zeta^4)^2 \\ \hline
\end{array}
\]
\end{table}

We may therefore compute the points on each line where $I_f$ vanishes; the resulting points are listed in the table below.

\begin{table}[h!]
\[
\begin{array}{|l|l|} \hline
\text{Lines} & \text{Points where $I_f$ vanishes} \\ \hline
a_i = 3m_i + 3/2 & a_{i+1} = 3m_{i+1} + \{0, 3/2, 2, 5/2\} \\ \hline
a_i = 3m_i + 5/2 & a_{i+1} = 3m_{i+1} + \{1, 3/2, 2, 5/2\} \\\hline
a_{i} - a_{i-1} = 3 m_{i-1} + 3/2 & a_{i} = 3m_{i} + \{0, 1/2, 1, 3/2, 2, 5/2\} \\ \hline
\end{array}
\]
\end{table}

These are all intersection points of the lines, meaning that $I_f$ is non-trivial for points lying only on a single line.  At such points, the nullspace is one-dimensional and generated by a vector with entries given in Table \ref{Nsigma}. This means that there is only one possible $\rho$ for which $L_a(\rho)$ is finite dimensional; in particular, it is the representation $\rho$ of lowest weight such that the corresponding entry is non-zero.  Thus, we may determine the lowest weight $\tau$ of $M$.  Further, because $n_\tau = 1$, this allows us to determine the only possible values of $N$. We will analyze these values for each set of lines separately.  

First consider the lines $a_i = 3m + 3/2$ and $a_i = 3m + 5/2$.  We have from Table \ref{Nsigma} that $[L_a(\tau)] = \sum_\sigma n_{\sigma} [M_a(\sigma)]$ with $n_{\CC_{-i}} = n_{\CC_{2 - i}} = e$, $n_{V_{1 - i}} = -e$, and all other $n_\sigma = 0$ for some $e \in \{+ 1, -1\}$.  Thus, we see that $\tau = \CC_{-i}$, $\tau = \CC_{2 - i}$, or $\tau = V_{1 - i}$.   Recall now that we computed 
\[
h(\CC_{-i}) = \frac{4}{3}(a_{i - 1} - a_i) + 1, \qquad h(\CC_{2 - i}) = \frac{4}{3}(a_i - a_{i + 1}) + 1, \qquad h(V_{1 - i}) = \frac{2}{3}(a_{i-1} - a_{i+1}) + 1.
\]
Notice that $h(V_{1 -i}) = \frac{1}{2} [h(\CC_{-i}) + h(\CC_{2 - i})]$, meaning that $V_{1 - i}$ cannot be the lowest weight unless $a_i = 0$, which does not occur in this case.  Hence, we conclude that $e = 1$.  Now, we have that 
\[
h(\CC_{-i}) < h(\CC_{2 -i}) \iff a_{i-1} + a_{i+1} < 2 a_i \iff 0 < a_i,
\]
meaning that $\tau = \CC_{-i}$ if $a_i > 0$ and $\tau = \CC_{2 - i}$ if $a_i < 0$.  By Equation (\ref{char}), we may then compute the character of $M = L_a(\tau)$ as
\[
\chi_{M}(g, t) = \chi_{M_a(\CC_{-i})}(g, t) + \chi_{M_a(\CC_{-i+2})}(g, t) - \chi_{M_a(V_{-i+1})}(g, t).
\]

Now consider the lines $a_{i} - a_{i-1} = 3m + 3/2$.  From Table \ref{Nsigma}, we may write $[L_a(\tau)] = \sum_\sigma n_{\sigma} [M_a(\sigma)]$ for $n_{\CC_{-i}} = n_{V_{-i}} = e$, $n_{\CC^3} = -e$, and all other $n_\sigma = 0$ for some $e \in \{+1, -1\}$.  Hence, we have that $\tau = \CC_{-i}$, $\tau = V_{-i}$ or $\tau = \CC^3$.  Notice that 
\[
h(\CC_{-i}) = \frac{4}{3}(a_{i-1} - a_{i}) + 1, \qquad h(V_{-i}) = \frac{2}{3}(a_{i} - a_{i-1}) + 1, \qquad h(\CC^3) = 1,
\]
meaning that $h(\CC^3) = \frac{1}{3} h(\CC_{- i}) + \frac{2}{3} h(V_{ - i})$, hence $\tau \neq \CC^3$.  This implies that $e = 1$.  Notice further that 
\[
h(\CC_{- i}) < h(V_{- i}) \iff  \frac{4}{3}(a_{i-1} - a_{i}) < \frac{2}{3}(a_{i} - a_{i-1}) \iff 2 (a_{i} - a_{i-1}) > 0,
\]
meaning that we must have $\tau = \CC_{ - i}$ if $2 (a_{i} - a_{i-1}) > 0$ and $\tau = V_{ - i}$ otherwise.  We may compute the character of $M$ in either case as
\[
\chi_{M}(g, t) = \chi_{M_a(\CC_{-i})}(g,t) + \chi_{M_a(V_{-i})}(g,t) - \chi_{M_a(\CC^3)}(g,t).
\]
Hence, in both cases, $M$ has the desired character, finishing the proof.
\end{proof}

To show that these representations exist, we use a result of Rouquier.  Adopting the notations of \cite{Rou}, we find that $(h_1, h_2) = (a_0/3, -a_2/3)$, $(x_0, x_1, x_2) = (1, e^{\frac{2 \pi i}{3}(a_0 - 1)}, e^{\frac{2 \pi i}{3}(1 - a_2)})$, and $K_0 = \QQ(e^{\frac{2 \pi i}{12}})$.  With these parameters, we have the following.
\begin{lemma}[Proposition 5.14 in \cite{Rou}] \label{rouquier}
Suppose all $x_i$'s are distinct with finite order.  Let $r \in \NN$ be such that there is an automorphism $\sigma$ of $K_0(\{x_i\})/K_0$ with $\sigma(x_i) = x_i^r$.  Then, we have
\[
\chi_{L_{rh}(\tau)}(g, t) = \frac{\det_{V^*}(1 - g t^r)}{\det_{V^*}(1 - gt)} \chi_{L_h(\tau)}(g, t^r).
\]
In particular, $L_{rh}(\tau)$ is finite dimensional if and only if $L_h(\tau)$ is finite dimensional.
\end{lemma}
We now construct a family of points for which we may apply Lemma \ref{rouquier}.  First, notice that the $x_i$'s are distinct if $a_i \not \equiv 1 \pmod{3}$, so the first condition of Lemma \ref{rouquier} applies away from a discrete family of lines.   Now, if $x_i$ has finite order, denote this order by $o_i$.  Consider the set of points $(a_0, a_1, a_2)$ such that (1) $a_i \not\equiv 1 \pmod{3}$ for any $i$ and (2) $o_1$ and $o_2$ are finite; call these points \textit{good}.  If we have further that $\{\gcd(o_1, 12), \gcd(o_2, 12)\} = \{1, 2\}$, call the point \textit{great}.  Note here that any great point is good. 

At good points, we see that $K_0({x_i}) = Q(\eta)$ for some root of unity $\eta$.  In this case, any automorphism $K_0({x_i})/K_0$ sends $\eta$ to $\eta^s$ for some positive integer $s$ with $s \equiv 1 \pmod{12}$ (because $e^{\frac{2 \pi i}{12}}$ is sent to $e^{\frac{2 \pi i}{12} s}$).  Therefore, for $r$ such that there is some fixed $s \equiv 1 \pmod{12}$ with $x_i^r = x_i^s$ for all $i$, the hypothesis of Lemma \ref{rouquier} are satisfied.  At good points, we may simply choose $r = s$, so we may apply Lemma \ref{rouquier} for all $r \equiv 1 \pmod{12}$.  At great points, it suffices to choose $r$ so that $r \equiv s \pmod{o_1}$ and $r \equiv s \pmod{o_2}$ with $s \equiv 1 \pmod{12}$; because $\{\gcd(o_1, 12), \gcd(o_2, 12)\} = \{1, 2\}$, we may apply Lemma \ref{rouquier} for all sufficiently large odd $r$ in this case.  We are now ready to prove Proposition \ref{exist}.

\begin{proof}[Proof of Proposition \ref{exist}]
We first claim that $L_a(\CC_{-i})$ is finite dimensional on  $a_i - a_{i-1} = 3/2$, $L_a(V_{-i})$ is finite dimensional on $a_i - a_{i-1} = -3/2$, $L_a(\CC_{-i+2})$ is finite dimensional on $a_i = -1/2$, and $L_a(\CC_{-i})$ is finite dimensional on $a_i = 5/2$. Take $\widehat{M}_a(\tau) = \tau^* \underset{\CC[W] \ltimes SV^*} \otimes H_a(W)$ and let $M_a(\tau)^\vee = \widehat{M}_a(\tau)^*$ be its restricted dual.  Then, we obtain a natural map $\nu_\tau: M_a(\tau) \to M_a(\tau)^\vee$; by Lemma 11.6 of \cite{E1}, the image of this map is $L_a(\tau)$.  Notice that if $\nu_\tau$ vanishes in grade $d$, then it must vanish in all higher grades. Therefore, for a fixed $\tau$, $L_a(\tau)$ is finite dimensional with highest weight in grade at most $d-1$ if and only if $\Imm(\nu_\tau)$ is trivial in grade $d$. Computing the value of $\Imm(\nu_\tau)$ in each grade in Magma, we see that $\nu_{\CC_{-i}}$ vanishes in grade $2$ on $a_i - a_{i-1} = 3/2$, $\nu_{V_{-i}}$ vanishes in grade $2$ on $a_i - a_{i-1} = -3/2$, $\nu_{\CC_{-i+2}}$ vanishes in grade $1$ on $a_i = -1/2$, and $\nu_{\CC_{-i}}$ vanishes in grade $9$ on $a_i = 5/2$. This gives the claim.

We will use Lemma \ref{rouquier} to deduce the existence of more representations from the ones we have just found.  We first claim that great points are Zariski dense on the lines $a_1 - a_{0} = 3/2$, $a_1 - a_{0} = -3/2$, $a_1 = 5/2$, and $a_1 = -1/2$. This amounts to showing that there are infinitely many great points on each of the lines; we give explicit constructions of these parametrized by an odd prime $p \geq 7$ in Table \ref{great22}.
\begin{table}[h!]
\caption{Great points on the lines \label{great22}}
\begin{tabular}{|l|l|} \hline
Line & $(a_{0}, a_1, a_{2})$ \\ \hline
$a_1 - a_{0} = 3/2$ & $\left(- \frac{7p - 24}{2p}, \frac{5p - 12}{2p}, \frac{p-6}{p}\right)$ \\ 
$a_1 - a_{0} = -3/2$ & $\left(- \frac{p - 24}{2p}, -\frac{p + 12}{2p}, \frac{p-6}{p}\right)$ \\
$a_1 = 5/2$ & $\left(- \frac{7p - 12}{2p}, \frac{5}{2}, \frac{p-6}{p}\right)$ \\
$a_1 = -1/2$ & $\left(- \frac{p - 12}{2p}, -\frac{1}{2}, \frac{p-6}{p}\right)$ \\ \hline
\end{tabular}
\end{table}

We now claim that for each line $\ell$, $L_a(\tau)$ is finite dimensional on a Zariski dense subset of $\ell$.  We first do this for the lines $a_1 - a_{0} = 3m + 3/2$, $a_1 =3m + 3/2$, and $a_1 = 3m + 3/2$.  We consider each family of these lines in turn.

First, consider the lines $a_1 - a_{0} = 3m + 3/2$ for $m \geq 0$. For one of the constructed great points $a$ on the line $a_i - a_{i-1} = 3/2$, consider the corresponding point $a' = (2m + 1)a$ on the line $a_1 - a_{0} = 3m + 3/2$.  Notice that $a'$ has non-integral rational coordinates and hence is good for all but finitely many $a$.  Now, choose $r$ odd and sufficiently large so that $r = (2m + 1) r'$, where $r' \equiv 1 \pmod{12}$.  By Lemma \ref{rouquier} with parameter $r$, we see that $L_a(\CC_{2})$ is finite dimensional at the point $r \cdot a = r' \cdot a'$.  Applying Lemma \ref{rouquier} again with parameter $r'$, we find that $L_a(\CC_{2})$ is finite dimensional at $a'$, meaning that $L_a(\CC_{2})$ is finite dimensional on a Zariski dense subset of the line $a_1 - a_{0} = 3m + 3/2$.

Now, consider the lines $a_1 - a_{0} = -3m - 3/2$ for $m \geq 0$.  For any great point $a$ on the line $a_1 - a_{0} = -3/2$, consider $a'= (2m + 1)a$ on the line $a_1 - a_{0} = - 3m - 3/2$.  As before, $a'$ is good for all but finitely many $a$.   Now, choose $r$ odd and sufficiently large so that $r = (2m + 1) r'$, where $r' \equiv 1 \pmod{12}$.  By Lemma \ref{rouquier} with parameter $r$, we see that $L_a(V_{2})$ is finite dimensional at the point $r \cdot a = r' \cdot a'$.  Applying Lemma \ref{rouquier} again with parameter $r'$, we find that $L_a(V_{2})$ is finite dimensional at $a'$, meaning that $L_a(V_{2})$ is finite dimensional on a Zariski dense subset of the line $a_1 - a_{0} = -3m-3/2$.

Now, consider the lines $a_1 = 3m + 3/2, 3m + 5/2$ for $m \geq 0$. For any great point $a$ on the line $a_1 = 5/2$, consider $a' = \frac{k}{5}a$ lying on one of the lines $a_1 = 3m + 3/2$ or $a_1 = 3m + 5/2$ for some integer $k$.  Note again that $a'$ is good for all but finitely many $a$.  Choose $r$ odd and sufficiently large so that $5r = k r'$ with $r' \equiv 1 \pmod{12}$.  Applying Lemma \ref{rouquier} twice with parameters $r$ and $r'$, we see that $L_a(\CC_1)$ is finite dimensional at $ra = \frac{k}{5} r' a$ and hence $a'$.  This shows that $L_a(\CC_1)$ is finite dimensional on Zariski dense subsets of the lines $a_1 = 3m + 3/2$ and $a_1 = 3m + 5/2$.

Finally, consider the lines $a_1 = -3m -1/2, -3m - 3/2$ for $m \geq 0$. For any great point $a$ on the line $a_1 = -1/2$, consider $a' = ka$ lying on one of the lines $a_1 = -3m -1/2$ or $a_1 = -3m -3/2$ for some integer $k$.  Again, $a'$ is good for all but finitely many $a$.  Choose $r$ odd and sufficiently large so that $r = k r'$ with $r' \equiv 1 \pmod{12}$.  Applying Lemma \ref{rouquier} twice with parameters $r$ and $r'$, we see that $L_a(\CC_{2})$ is finite dimensional at $ra = \frac{k}{5} r' a$ and hence $a'$.  This shows that $L_a(\CC_2)$ is finite dimensional on a Zariski dense subset of the lines $a_1 = -3m - 1/2$ and $a_1 = 3m - 3/2$.

Now, recalling that the functor $\omega$ of Lemma \ref{rotate} is an equivalence of categories $\OO(H_{a_0, a_1, a_2}(W)) \to \OO(H_{a_2, a_0, a_1})(W)$, we may apply it to the previous families of lines to obtain the conclusion for all lines. Hence, we have now shown that for each line $\ell$, the corresponding representation $L_a(\tau)$ is finite dimensional on a Zariski dense set $\ell'' \subset \ell$. Let $\ell'$ be the set of points in $\ell''$ that lie only on a single line; notice that $\ell'$ is still Zariski dense in $\ell$. Recall that $L_a(\tau) = M_a(\tau)/ J_a(\tau)$, where $J_a(\tau)$ is the sum of all proper standard submodules of $M_a(\tau)$.  Now, for $a \in \ell'$, by Proposition \ref{unique}, the character of $L_a(\tau)$ is given by 
\[
\chi(g, t) = \chi_{M_a(\CC_{-i})}(g,t) + \chi_{M_a(V_{-i})}(g,t) - \chi_{M_a(\CC^3)}(g,t)
\]
on the lines $a_i - a_{i-1} = 3m + 3/2$ and by
\[
\chi(g,t) = \chi_{M_a(\CC_{-i})} + \chi_{M_a(\CC_{2-i})}(g,t) - \chi_{M_a(V_{1 - i})}(g,t)
\]
on the lines $a_i = 3m + 3/2, 3m + 5/2$.  In each case, $J_a(\tau)$ is the image of a single map $M_a(\sigma) \to M_a(\tau)$ for some representation $\sigma$ of $W$.  In particular, for each $a \in \ell'$, $M_a(\tau)$ must contain a unique singular copy of $\sigma$ in a fixed grade $d$, where the value of $d$ can be determined from the character of $L_a(\tau)$.

Suppose now that $M_a(\tau)$ has $k$ copies of $\sigma$ in grade $d$; let $X$ be the subspace of $M_a(\tau)$ spanned by these and fix an isomorphism $X \simeq \sigma^{\oplus k}$.  Write $\iota_i: \sigma \to X$ for the $i$th inclusion. Let $\sigma_{b_1, \ldots, b_k}$ be the image of the embedding $\iota_{b_1, \ldots, b_k}: \sigma \to X$ given by $v \mapsto \sum_i b_i \iota_i(v)$.  Fixing bases $\{y_1, y_2\}$ for $V^*$ and $\{v_i\}$ for $\sigma$, notice that $\sigma_{b_1, \ldots, b_k} \subset X$ consists of singular vectors if and only if $(b_1, \ldots, b_k)$ is in the nullspace of the map $\phi_a: \CC^k \to M_a(\tau)[n - 1]^{\oplus 2 \dim \sigma}$ given by
\[
(b_1, \ldots, b_k) \mapsto \left\{\sum_m b_m \cdot (y \cdot \iota_m(v_i))\right\}_{ij}.
\]
For $a \in \ell'$, there is a unique singular copy of $\sigma$ in grade $d$, so $\phi_a$ has nullspace of dimension $1$ on a Zariski dense set in $\ell$.  But the matrix entries of $\phi_a$ are polynomial in $a$, so generic semi-continuity implies that $\phi_a$ has nullspace of dimension $1$ on all but finitely many points $\{p_1, \ldots, p_n\}$ in $\ell$.  Hence, we may define a regular map $\psi:\ell - \{p_1, \ldots, p_n\} \to \PP^{k-1}$ sending $a$ to the nullspace of $\phi_a$; $\psi$ extends uniquely to some $\overline{\psi}$ on all of $\ell$.

By continuity, we see that $\sigma_{\overline{\psi}(a)}$ is singular for all $a \in \ell$, so we have a submodule $I_a(\tau)$ generated by $\sigma_{\overline{\psi}(a)}$. For generic $a$, $I_a(\tau)[t]$ is the image of the map $\rho_{a, t}: SV^*[t - d]^{\oplus \dim \sigma} \to M_a(\tau)[t]$ given by 
\[
\rho_{a, t}(g_1, \ldots, g_{\dim \sigma}) = \sum_i g_i \iota_{\psi(a)}(v_i).
\]
The matrix entries of $\rho_{a,t}$ are polynomial in $c$, so by generic semi-continuity, for any $a^* \in \ell'$, $\dim I_a(\tau)[t] \geq \dim I_{a^*}(\tau)[t]$ for all but finitely many $a$.  Hence, for all but countably many $a$, we have that $\dim M_a(\tau)/I_a(\tau) \leq \dim M_{a^*}(\tau)/I_{a^*}(\tau)$ is finite dimensional.  Let $\{q_i\}$ be the at most countable set of points where $\dim M_a(\tau)/I_a(\tau)$ is not finite dimensional.

Let $m = \dim M_{a}(\tau)/I_{a}(\tau)$ for generic $a \in \ell$ and let $N = \dim M_a(\tau)^{\leq m + 1}$.  We may now define a map $\eta: \ell - \{q_i\} \to \Gr_{N - m}(M_a(\tau)^{\leq m + 1})$ by $a \mapsto I_a(\tau) \cap M_a(\tau)^{\leq m + 1}$.  Here, we are using the fact that $M_a(\tau)^{> m} \subset I_a(\tau)$. Notice that $\eta$ is polynomial on $\ell$, hence has an extension to all of $\ell$.  Now, the action of any $x \in H_a(W)$ on $\widehat{I}_a(\tau)$ is continuous in $a$, so $\widehat{I}_a(\tau) := \eta(a) \oplus M_a(\tau)^{> m}$ is a submodule of $M_a(\tau)$ for all $a \in \ell$.  Hence, for all $a \in \ell$, we see that $M_a(\tau)/\widehat{I}_a(\tau)$ is a finite dimensional representation of $H_a(W)$.  

It now remains for us to show that the character $\chi_a(g, t)$ of $M_a(\tau)/\widehat{I}_a(\tau)$ is as desired.  Note that $\chi_a(g, t)$ can have a non-zero coefficient of $t^q$ only for $q \in h(\tau) + \ZZ_{\geq 0}$.  Write $[t^q] \chi_a(g, t)$ for the coefficient of $t^q$ in $\chi_a(g, t)$.  For fixed $g$, we claim that, for all $q \in h(\tau) + \ZZ_{\geq 0}$, the function $a \mapsto [t^q] \chi_a(g, t)$ is continuous for $a \in \ell$.  But notice that $[t^q] \chi_a(g, t)$ is the character of $g$ on the subspace of $M_a(\tau)/\widehat{I}_a(\tau)$ corresponding to $t^q$.  This is a regular family of representations of $W$ parametrized by $a$, hence its character $[t^q] \chi_a(g, t)$ is continuous in $a$. 

Now, let $\chi^*_a(g, t)$ be the desired character for $a \in \ell$.  For points not on intersections of lines, notice that $M_a(\tau)/\widehat{I}_c(\tau)$ is the unique finite dimensional representation, so by Proposition \ref{unique}, we have $\chi_a(g, t) = \chi^*_a(g, t)$ for $a$ on $\ell$ only.  This means that $[t^q]\chi_a(g, t) = [t^q]\chi^*_a(g,t)$ for all $q$ and $a$ only on $\ell$.  By continuity, we see that $[t^q]\chi_a(g, t) = [t^q]\chi^*_a(g, t)$ for all $a \in \ell$, meaning that 
$\chi_a(g, t) = \chi_a^*(g, t)$ for all $a \in \ell$.  Thus $M_a(\tau)/\widehat{I}_a(\tau)$ has the desired character, finishing the proof.
\end{proof}

We have now proved Theorem \ref{lines}, showing that, for each point $a$ on exactly one line, there is a unique irreducible finite dimensional representation $Y_a$.  Now, by Proposition 1.12 in \cite{BEG}, we have that $\Ext^1(Y_a, Y_a) = 0$, meaning that the category of finite dimensional representations of $H_a(W)$ is semi-simple at these points.  This gives a complete description of finite dimensional representation theory of $H_a(W)$ away from intersection points of the lines.  

\subsection{At intersection points}

At intersection points, the irreducible representations coming from each incident line may degenerate, creating a more complicated representation theory.  Our approach will be to use the functors $\delta: M \mapsto M^*$ and $\omega: M \mapsto M \otimes \CC_2$ given in Lemmas \ref{flip} and \ref{rotate} to reduce our analysis to a smaller number of points.  At these points, we will use Lemma \ref{poletrick} and Proposition \ref{exist} to determine the structure of the representations.

We will use the following notations.  As in Proposition \ref{unique}, take $N = (n_{\CC_0}, n_{\CC_1}, n_{\CC_2}, n_{V_0}, n_{V_1}, n_{V_2}, n_{\CC^3})$.  Now, for any such $N$, define the character $\chi_{N}(g,t) = \sum_\sigma n_\sigma \chi_{M_a(\sigma)}(g,t)$.  Further, for any $\tau$ with the decomposition $[L_a(\tau)] = \sum_\sigma n_{\tau\sigma} [M_a(\sigma)]$, define the vector $N_\tau = (n_{\tau \CC_0}, n_{\tau\CC_1}, n_{\tau\CC_2}, n_{\tau V_0}, n_{\tau V_1}, n_{\tau V_2}, n_{\tau\CC^3})$.  Now, for any finite dimensional representation $M$ of $H_c(\tau)$, let $\grade(M) = [h_\text{min}(M), h_\text{max}(M)]$ be the pair of lowest and highest $h$-weights of $M$.

\begin{proof}[Proof of Theorem \ref{points}]
Notice first that Lemmas \ref{flip} and \ref{rotate} together mean that the categories $\OO(H_a(W))$ are equivalent for $c$ corresponding to all permutations of $(a_0, a_1, a_2)$.  It therefore suffices for us to prove the theorem for a suitable permutation of the parameters.

Let $a$ be an intersection point of the lines and $M$ a finite dimensional irreducible representation of $H_a(W)$.  Write $[M] = \sum_\sigma n_\sigma [M_a(\sigma)]$ by Lemma \ref{basis}.  Take $y = \zeta^{a_0/2}$ and $w = \zeta^{a_1/2}$.  Then, Lemma \ref{poletrick} gives us a matrix $A$ with coefficients in $\CC(y^{\frac{1}{3}}, w^{\frac{1}{3}})$ such that all possible values of $N$ are in the nullspace of $A$.  At each intersection point, we will find that the nullspace of $A$ has a basis consisting of vectors $N$ such that $\chi_N$ are characters of finite dimensional representations $M_N$ (induced by the lines incident to $a$).

We will examine each type of point separately.  First, we note the following situations, which will occur frequently in our analysis. 
\begin{lemma} \label{irred}
Let $M$ be a finite dimensional $H_a(W)$-module with lowest (highest) weight $\tau$.  Suppose that there is a basis for the nullspace of $A$ consisting of vectors $N$ such that each $\chi_N$ is the character of some finite dimensional module $M_N$ and that the $M_N$ have distinct lowest (highest) weights that are at most (at least) $h(\tau)$.  Then, $M$ is irreducible.
\end{lemma}
\begin{proof}
Suppose for contradiction that $M$ were not irreducible; then, we can find some irreducible submodule $M'$ of $M$ with lowest weight $\tau$.  Let $M \onto M'$ be the corresponding surjection and $M''$ its kernel, which has higher lowest weight than $\tau$. Because $M$ and $M'$ are finite dimensional, their characters are given by $\chi_{N_M}$ and $\chi_{N_{M'}}$ for some $N_M$ and $N_{M'}$ in the nullspace of $A$; therefore, the character of $M''$ is $\chi_{N_M - N_{M'}}$.  Both $N_M$ and $N_{M'}$ are in the span of the given basis vectors $\{N_i\}$ for the nullspace of $A$, hence so is $N_{M} - N_{M'}$.  Because the $M_N$ have distinct lowest weights for $N$ in the basis, this means that $N_M - N_{M'}$ contains a term with lowest (highest) weight at most (at least) $h(\tau)$, a contradiction. So $M$ must be irreducible.
\end{proof}

\begin{lemma} \label{singular}
If $a_2 = 3m + 2$ for some $m$, then $\dim \Hom_{H_a(W)}(M_a(\CC_0), M_a(\CC_1)) = 1$ if $m > 0$ and $\dim \Hom_{H_a(W)}(M_a(\CC_1), M_a(\CC_0)) = 1$ if $m < 0$.
\end{lemma}
\begin{proof}
First, notice that because $\widehat{\CC}_0$ and $\widehat{\CC}_1$ are always irreducible, we have in all cases that
\[
\dim \Hom_{H_a(W)}(M_a(\CC_0), M_a(\CC_1)) \leq \dim \Hom_{\HH(\{q_{H, j}^*\})}(\widehat{\CC}_0, \widehat{\CC}_1) \leq 1.
\]
Similarly, we see that $\dim \Hom_{H_a(W)}(M_a(\CC_1), M_a(\CC_0)) \leq 1$.  So it suffices to show that $M_a(\CC_0)$ and $M_a(\CC_1)$ contain singular copies of $\CC_1$ and $\CC_0$, respectively, for the given values of $a_2$.  

For $H \in \AAA$ with $H = H_s$ for some $s \in S$, take $\alpha_H = \alpha_s$.  Let $P = \prod_{H \in \AAA} \alpha_H$; note that $P$ is a fourth degree polynomial in $\CC[V]$.  We claim that for $a_2 = 3m + 2$, $P^{-3m+1}$ is singular in $M_a(\CC_0)$ if $m < 0$ and $P^{3m+2}$ is singular in $M_a(\CC_1)$ if $m \geq 0$.  For $y \in V$, let
\[
D_y = \partial_y - \sum_{s \in S} c_s (y, \alpha_s)(x, \alpha_s^\vee) s
\]
denote the Dunkl operator corresponding to $y$.  Then we may write
\[
D_y P = \partial_y P - \sum_{s \in S_1} c_1 (y, \alpha_s)(x, \alpha_s^\vee) s - \sum_{s \in S_2} c_2 (y, \alpha_s)(x, \alpha_s^\vee) s,
\]
where $S_1$ consists of reflections with $\lambda_s = \omega$ and $S_2$ consists of reflections with $\lambda_s = \omega^2$. Define the quantities $A = \partial_y P$, $B = \sum_{s \in S_1} c_1 (y, \alpha_s)(x, \alpha_s^\vee) s$, and $C = \sum_{s \in S_2} c_2 (y, \alpha_s)(x, \alpha_s^\vee) s$.  For any $s \in S$, notice that
\[
s \cdot P = \prod_{H \in \AAA} (s \cdot \alpha_H) = \lambda_s \alpha_{H_s} \prod_{H \neq H_s} (s \cdot \alpha_H) = \lambda_s P,
\]
where we note that $s$ acts by the identity on $\prod_{H \neq H_s} \alpha_H$.  Therefore, we find that $s \cdot P = \omega P$ for $s \in S_1$ and $s \cdot P = \omega^2 P$ for $s \in S_2$.  Then, notice that in $M_a(\CC_0)$, we have
\[
D_y P^k = [k A + B(1 + \omega + \cdots + \omega^{k-1}) + C (1 + \omega^2 + \cdots + \omega^{2k - 2})] P^{k-1}
\]
and that in $M_a(\CC_1)$, we have
\[
D_y P^k = [k A + B(\omega + \omega^2 + \cdots + \omega^{k}) + C (\omega^2 + \omega^4 + \cdots + \omega^{2k})]P^{k-1}.
\]
In each case, computing the values of $A$, $B$, and $C$ in Magma, we find that $D_y P^k = 0$ for all $y$ if and only if $k = -3m + 1$ with $m < 0$ for $M_a(\CC_0)$ or $k = 3m + 2$ with $m \geq 0$ for $M_a(\CC_1)$. This establishes the claim. 
\end{proof}
\begin{remark}
We may also obtain Lemma \ref{singular} using Theorem 3.2 in \cite{BE} and the classification of finite dimensional irreducible representations of $H_a(\ZZ/3\ZZ)$.
\end{remark}

We now analyze each type of intersection point separately.

\noindent \textbf{Points of type (11):} Let $a = (a_0, a_1, a_2)$ be a point of type (11). Then, we have that $a_i - a_{i-1} = 3m_i + 3/2$ and $a_{i + 1} - a_{i} = 3 m_{i+1} + 3/2$ for some $i$ and $m_i, m_{i+1} \in \ZZ$.  Assume first that $i = 2$; that is, we have that $a_2 - a_1 = 3 m_2 + 3/2$ and $a_0 - a_2 = 3 m_0 + 3/2$.  This means that $(a_0, a_1, a_2) = (2 m_0 + m_2 + 3/2, -2m _2 - m_0 - 3/2, m_2 - m_0)$.  Since $a$ is not a point of type (1122), we must have that $a_0 - 1/2, a_1 - 1/2 \equiv 0 \pmod{3}$.  This means in particular that $m_0 \neq m_2 \pmod{3}$ and that $4m_0 + 2m_2 + 3$ and $2m_0 + 4m_2 + 3$ are non-zero.

Now, we find that that $(y, w) = (\zeta^{4m_0 + 2m_2 + 3}, \zeta^{-4m _2 - 2m_0 - 3})$.   For $m_0, m_2 \in \ZZ$, this gives a finite number of possible values for $(y, w)$.  Computing in Magma, we find that, for each of these values of $(y, w)$, $A$ has a two dimensional nullspace spanned by $N_0 = (1, 0, 0, 1, 0, 0, -1)$ and $N_1 = (0, 1, 0, 0, 1, 0, -1)$.  Hence, there must be exactly two finite dimensional irreducible representations at these points.

By Proposition \ref{exist}, there are finite dimensional representations $S_0$ and $S_1$ with characters $\chi_{N_0}$ and $\chi_{N_1}$ at this point.  Let the lowest weights of $S_0$ and $S_1$ be $\tau_0$ and $\tau_1$, respectively, where $\tau_i \in \{\CC_i, V_i\}$ depends on $a$.  Then, we see that $L_a(\tau_0)$ and $L_a(\tau_1)$ must be the two finite dimensional representations.  It remains to compute their characters.  We now consider different cases for the sign of $m_0$ and $m_2$.
\begin{itemize}
\item First, suppose that $m_2, m_0 \geq 0$. Then, we have that $(\tau_0, \tau_1) = (\CC_0, \CC_1)$.  In this case, we have that $\grade(S_0) = [h(\CC_1) + 4m_2 - 4m_0, h(\CC_1) + 4m_2 + 2m_0 + 1]$ and $\grade(S_1) = [h(\CC_1), h(\CC_1) + 6m_2 + 1]$.  

If $m_2 > m_0$, $S_0$ is irreducible by Lemma \ref{irred}.  Now, $L_a(\tau_1)$ is a quotient of $S_1$, hence its character must be given by $\chi_{N_1} - k \chi_{N_0} = \chi_{(-k, 1, 0, -k, 1, 0, -1+k)}$ for some $k \geq 0$.  By Lemma \ref{singular}, we see that $\dim \Hom_{H_a(W)}(M_a(\CC_0), M_a(\CC_1)) = 1$ in this case.  Therefore, there is exactly one singular copy of $\CC_0$ in $M_a(\CC_1)$.  But $h(\CC_0) < h(\CC^3)$, so $\CC_0$ is the singular vector of lowest weight in $M_a(\CC^1)$, hence it is singular in $S_1$ as well.  This implies that $k = 1$ and hence $L_a(\tau_1)$ has character $\chi_{(-1, 1, 0, -1, 1, 0, 0)}$.  

Further, we obtain a map $M_a(\tau_0) \to S_1$; but the unique singular representation $\CC^3 \subset M_a(\tau_0)$ maps to zero in $S_1$ by the construction of $S_1$, meaning that this map factors through $M_a(\tau_0) \to L_a(\tau_0) \to S_1$ to give rise to a map $S_0 \to S_1$.  Notice now that $\coker(S_0 \to S_1)$ is a non-zero submodule of $L_a(\tau_1)$, hence is the other irreducible representation.  If $m_2 < m_0$, $S_1$ is irreducible by Lemma \ref{irred}, and an analogous argument again using the other portion of Lemma \ref{singular} shows that $\coker(S_1 \to S_0) = L_a(\tau_0)$ has character $\chi_{(1, -1, 0, 1, -1, 0, 0)}$.

To summarize, if $m_2 > m_0$, $S_0$ and $\coker(S_0 \to S_1)$ are the irreducible representations with characters $\chi_{N_0} = \chi_{(1, 0, 0, 1, 0, 0, -1)}$ and $\chi_{N_1 - N_0} = \chi_{(-1, 1, 0, -1, 1, 0, 0)}$.  If $m_2 < m_0$, $S_1$ and $\coker(S_1 \to S_0)$ are the irreducible representations with characters $\chi_{N_1} = \chi_{(0, 1, 0, 0, 1, 0, -1)}$ and $\chi_{N_0 - N_1} = \chi_{(1, -1, 0, 1, -1, 0, 0)}$.

\item Now, suppose that $m_2 \geq 0$, $m_0 < 0$.  Then, we have that $(\tau_0, \tau_1) = (V_0, \CC_1)$ with $\grade(S_0) = [h(\CC_1) + 4m_2 + 2m_0 + 3, h(\CC_1) + 4m_2 - 4m_0 - 2]$ and $\grade(S_1) = [h(\CC_1), h(\CC_1) + 6 m_2 + 1]$.  If $4m_0 + 2m_2 + 3$ and $2m_0 + 4m_2 + 3$ have different signs, we see that $S_0$ and $S_1$ are representations with either highest lowest weight or lowest highest weight in the set $\{S_0, S_1\}$, so by Lemma \ref{irred} they are both irreducible.

If $4m_0 + 2m_2 + 3$ and $2m_0 + 4m_2 + 3$ are both positive, then $S_0$ is irreducible by Lemma \ref{irred}.  Now, we see that $L_a(\tau_1)$ is a quotient of $S_1$, hence its character must be given by $\chi_{N_1} - k \chi_{N_0} = \chi_{(-k, 1, 0, -k, 1, 0, -1 + k)}$ for some $k \geq 0$.  If $k > 0$, then we have $h(\CC_1) < h(V_0) < h(\CC^3) < h(\CC_0) < h(V_1)$, meaning that $V_0$ must be singular in $\CC_1$.  But by Lemma \ref{mappings2}, this is impossible because $a_1$ is of the form $3m + 1/2$.  Therefore, $k = 0$, meaning that $S_1$ is irreducible.  If $4m_0 + 2m_2 + 3$ and $2m_0 + 4m_2 + 3$ are both negative, the same argument with the roles of $S_0$ and $S_1$ reversed implies again that both $S_0$ and $S_1$ are irreducible.

To summarize, we have that $S_0$ and $S_1$ are the irreducible representations of $H_a(W)$ in this case, and they have characters $\chi_{N_0} = \chi_{(1, 0, 0, 1, 0, 0, -1)}$ and $\chi_{N_1} = \chi_{(0, 1, 0, 0, 1, 0, -1)}$.

\item Now, suppose that $m_2 < 0$, $m_0 \geq 0$. Then, we have that $(\tau_0, \tau_1) = (\CC_0, V_1)$ with $\grade(S_0) = [h(\CC_0), h(\CC_0) + 6m_0 + 1]$ and $\grade(S_1) = [h(\CC_0) + 4m_0 + 2m_2 + 3, h(\CC_0) + 4m_0 - 4m_2 - 2]$.  As in the previous case, if $4m_0 + 2m_2 + 3$ and $2m_0 + 4m_2 + 3$ have different signs, $S_0$ and $S_1$ are representations with either highest lowest weight or lowest highest weight in the set $\{S_0, S_1\}$, so they are irreducible by Lemma \ref{irred}.  

If $4m_0 + 2m_2 + 3$ and $2m_0 + 4m_2 + 3$ are both positive, then $S_1$ is irreducible by Lemma \ref{irred}.  Now, we see that $L_a(\tau_0)$ is a quotient of $S_0$, hence its character must be given by $\chi_{N_0} - k \chi_{N_1} = \chi_{(1, -k, 0, 1, -k, 0, -1 + k)}$ for some $k \geq 0$.  If $k > 0$, then we have $h(V_1) < h(\CC_0) < h(\CC^3) < h(V_0) < h(\CC_1)$, meaning that $\CC_0$ must be singular in $V_1$.  But by Lemma \ref{mappings2}, this is impossible because $a_0$ is of the form $3m + 1/2$.  Therefore, $k = 0$, meaning that $S_0$ is irreducible.  If $4m_0 + 2m_2 + 3$ and $2m_0 + 4m_2 + 3$ are both negative, the same argument with the roles of $S_0$ and $S_1$ reversed implies again that both $S_0$ and $S_1$ are irreducible.

To summarize, we have that $S_0$ and $S_1$ are the irreducible representations of $H_a(W)$ in this case, and they have characters $\chi_{N_0} = \chi_{(1, 0, 0, 1, 0, 0, -1)}$ and $\chi_{N_1} = \chi_{(0, 1, 0, 0, 1, 0, -1)}$.

\item Finally, suppose that $m_2, m_0 < 0$. Then, we have that $(\tau_0, \tau_1) = (V_0, V_1)$.  Now, notice that $\delta$ gives a map between irreducible representations of $H_a(W)$ and $H_{a'}(W)$, where $a' = (-2m_2 - m_0 - 3/2, 2m_0 + m_2 + 3/2, m_2 - m_0) = (2m_0' + m_2' - 3/2, -2m_2'-m_0' +3/2, m_2' - m_0')$ for $m_0' = -m_2 - 1$ and $m_2' = -m_0 - 1$, where $m_0', m_2' \geq 0$.  Now, if $m_2 < m_0 \iff m_2' > m_0'$, then by the first case $\delta(S_0)$ and $\coker(\delta(S_0) \to \delta(S_1))$ are the irreducible representations of $H_{a'}(W)$.  Therefore, we see that $S_0$ and $\ker(S_1 \to S_0)$ are the irreducible representations of $H_a(W)$.  If $m_2 > m_0 \iff m_2' < m_0'$, then $\delta(S_1)$ and $\coker(\delta(S_1) \to \delta(S_0))$ are the irreducible representations of $H_{a'}(W)$, hence $S_1$ and $\ker(S_0 \to S_1)$ are the irreducible representations of $H_a(W)$.

To summarize, if $m_2 > m_0$, $S_1$ and $\ker(S_0 \to S_1)$ are irreducible with characters $\chi_{N_1} = \chi_{(0,1,0,0,1,0,-1)}$ and $\chi_{N_0 - N_1} = \chi_{(1, -1, 0,1,-1,0,0)}$.  If $m_2 < m_0$, $S_0$ and $\ker(S_1 \to S_0)$ are irreducible with characters $\chi_{N_0} = \chi_{(1,0,0,1,0,0,-1)}$ and $\chi_{N_1 - N_0} = \chi_{(-1,1,0,-1,1,0,0)}$.
\end{itemize}
Now consider $i \neq 2$.  The map $\omega$ given by Lemma \ref{rotate} maps the simple object $L_{a_0, a_1, a_2}(\tau)$ to a simple object in $\OO(H_{a_2, a_0, a_1}(W))$ with lowest weight $\tau \otimes \CC_2$.  Thus, we find that $\omega(L_{a_0, a_1,a_2}(\tau)) = L_{a_2, a_0, a_1}(\tau \otimes \CC_2)$, so $L_{a_2, a_0, a_1}(\tau \otimes \CC_2)$ is finite dimensional if and only if $L_{a_0, a_1, a_2}(\tau)$ is finite dimensional.  Because $\omega$ is exact, it induces a map $K(\OO(H_{a_0, a_1, a_2}(W))) \to K(\OO(H_{a_2, a_0, a_1}(W)))$.  This allows us to compute the character of $L_{a_2, a_0, a_1}(\tau \otimes \CC_2)$ using the fact that $\omega(M_{a_0, a_1, a_2}(\tau)) = M_{a_2, a_0, a_1}(\tau \otimes \CC_2)$. Applying $\omega$ either one or two times to the representations given in each case above gives the claimed list of finite dimensional irreducible representations.  This completes the analysis for points of type (11).

\noindent \textbf{Points of type (22):}  At these points, we have that $a_i = m_i + 1/2$ and $a_{i+1} = m_{i+1} + 1/2$ for some $i$ with $m_i, m_{i+1} \not \equiv 0 \pmod{3}$.  If $m_i \equiv m_{i+1} \pmod{3}$, then $a_{i+2} = - m_i - m_{i+1} - 1$.  In this case, we have $a_{i+1} - a_{i+2} = m_i + 2 m_{i+1} + 3/2$, so this point is also on a line of type $1$ and is not of type (22).  Hence, we must have that $m_i \not\equiv m_{i+1} \pmod{3}$.  

First consider the case $i = 0$. Then, we have $(a_0, a_1, a_2) = (m_0 + 1/2, m_1 + 1/2, -m_0 - m_1 - 1)$ and hence $(y, w) = (\zeta^{2m_0 + 1}, \zeta^{2m_1 + 1})$.  Again, we see that there are a finite number of possible values for $(y, w)$, and computing in Magma we find that, for each of these values, $A$ has a two dimensional nullspace spanned by $M_1 = (1, 0, 1, 0, -1, 0, 0)$ and $M_0 = (0, 1, 1, -1, 0, 0, 0)$.  Therefore, there are exactly two finite dimensional irreducible representations at these points.

By Proposition \ref{exist}, we have finite dimensional representations $R_1$ and $R_0$ with characters $\chi_{M_1}$ and $\chi_{M_0}$ at these points.  Let the lowest weights of $R_1$ and $R_0$ be $\rho_1$ and $\rho_0$, respectively, where the values of $\rho_1 \in \{\CC_0, \CC_2\}$ and $\rho_0 \in \{\CC_1, \CC_2\}$ depend on the location of $a$.  By Lemma \ref{weights}, $\rho_1 = \CC_0$ if $a_0 > 0$ and $\rho_1 = \CC_2$ if $a_0 < 0$; similarly, $\rho_0 = \CC_1$ if $a_1 < 0$ and $\rho_0 = \CC_2$ if $a_1 > 0$.  We divide now into cases depending on the sign of $m_0$ and $m_1$.  

\begin{itemize}
\item Suppose first that $m_0 < 0, m_1 \geq 0$.  Then, we have that $\rho_1 = \rho_0 = \CC_2$.   Notice now that $\grade(R_1) = [h(\CC_2), h(\CC_2) -4m_0 - 4]$ and $\grade(R_0) = [h(\CC_2),h(\CC_2) + 4m_1]$.  We claim that $-4m_0 - 4 \neq 4m_1$; suppose otherwise for contradiction.  Then, we have that $m_0 + m_1 = -1$.  But we have that  $m_0 - m_1, m_0, m_1 \not \equiv 0 \pmod{3}$, so $\{m_0, m_1\} \equiv \{1,2\} \pmod{3}$, meaning that $m_0 + m_1 \equiv 0 \pmod{3}$, a contradiction.  Therefore, one of $R_1$ and $R_0$ has lower highest weight than the other. Let this representation be $R_j$; if $m_0+m_1 + 1> 0$, we have $j = 1$; otherwise, if $m_0 + m_1 + 1 < 0$, we obtain $j = 0$.  By Lemma \ref{irred}, we have that $R_j = L_a(\CC_2)$.  

Let the module $R_{1-j}'$ be the kernel of the map $R_{1-j} \to L_a(\CC_2)$.  Notice that $R_{1-j}'$ is a finite dimensional representation with character $\chi_{M_{1-j}} - \chi_{M_{j}}$.   Suppose that $R_{1-j}'$ had lowest weight $\tau$.  Because $h(\tau) > h(\CC_2)$, we have $n_{\tau \CC_2} = 0$; further, we have $n_{\tau \tau} = 1$, because $N_\tau$ is in the integer span of $M_{1-j}$ and $M_j$, we find that $N_\tau = M_{1-j} - M_{j}$.  Thus we see that $R_{1-j}' = L_a(\tau)$, so $R_{j}$ and $R_{1-j}'$ are the two irreducible representations in this case and they have characters $\chi_{M_{j}}$ and $\chi_{M_{1-j}} - \chi_{M_{j}}$, respectively.

\item Now, suppose that $m_0 \geq 0, m_1 < 0$.  Then, we see that $\rho_1 = \CC_0$ and $\rho_0 = \CC_1$, meaning that $L_a(\CC_0)$ and $L_a(\CC_1)$ are the only finite dimensional irreducible representations.  Notice that $\grade(R_1) = [h(\CC_0), h(\CC_0) + 4m_0]$ and $\grade(R_0) = [h(\CC_0) + 4m_0 + 4m_1 + 4, h(\CC_0) + 4m_0]$.  Therefore, one of $R_1$ and $R_0$ has a higher lowest weight than the other and therefore is irreducible.  Let this be $R_k$; if $m_0 + m_1 + 1 > 0$, then $k = 0$; otherwise, if $m_0 + m_1 + 1 < 0$, then $k = 1$.  Now, by Lemma \ref{irred}, this means that $R_k$ is irreducible, so $L_a(\CC_{1 - k}) = R_k$. 

Applying the map $\delta$ of Lemma \ref{flip}, we see that $R_k$ and $R_{1 - k}$ have the same highest weight, as they had the same lowest weight in the previous case.  Therefore, we may define a map $R_{1 - k} \to L_a(\CC_{1 - k}) = R_k$; let $R_{1 - k}'$ be its kernel.  Notice that $R_{1 - k}'$ is finite dimensional with character $\chi_{M_{1 - k}} - \chi_{M_k}$.  Notice that $R_{1 - k}'$ has highest weight $\tau$ less than the common highest weight $\sigma$ of $R_k$ and $R_{1 - k}$, so applying Lemma \ref{irred} to $R_{1 - k}'$ and the basis $\{M_{1 - k} - M_k, M_k\}$ shows that $R_{1 - k}'$ is irreducible.  Therefore, the two irreducible representations in this case are $R_k$ and $R_{1 - k}'$, and they have characters $\chi_{M_k}$ and $\chi_{M_{1 - k}} - \chi_{M_k}$, respectively. 

\item Now, suppose that $m_0, m_1 \geq 0$; write $a_0 = m_0 + 1/2$ and $a_1 = m_1 + 1/2$ with $m_0, m_1 \not \equiv 0 \pmod{3}$.  In this case, $\rho_1 = \CC_0$ and $\rho_0 = \CC_2$, with $h(\CC_2) - h(\CC_0) = 4m_0 + 2 > 0$.  Now, from the form of its character, we see that $\grade(R_0) = [h(\CC_0), h(\CC_0) + 4m_0 + 1]$, while $R_1$ has lowest weight in grade $h(\CC_2) = h(\CC_0) + 4m_0 + 2$.  Because $R_0$ and $R_1$ exist in disjoint grades, no non-trivial linear combination of the characters of $R_0$ and $R_1$ can be the character of a quotient of either $R_0$ or $R_1$.  But the characters of $L_a(\CC_0)$ and $L_a(\CC_2)$ are in the integer span of the characters of $R_0$ and $R_1$, so we may conclude that $L_a(\CC_0) = R_0$ and $L_a(\CC_2) = R_1$ are the irreducible representations.

\item Finally, consider the case $m_0, m_1 < 0$; write $a_0 = -m_0 + 1/2$ and $a_1 = -m_1 + 1/2$ with $m_0, m_1 \not \equiv 0 \pmod{3}$.  We have that $\rho_1 = \CC_2$ and $\rho_0 = \CC_1$, with $h(\CC_2) - h(\CC_1) = 4m_1 - 2 > 0$.  Note that $\grade(R_1) = [h(\CC_1), h(\CC_1) + 4m_1 - 3]$, while $R_0$ has lowest weight in grade $h(\CC_2) = h(\CC_1) + 4m_1 - 2$.  As in the previous case, $R_0$ and $R_1$ have no non-trivial quotients whose characters are in the integer span of $\chi_{M_0}$ and $\chi_{M_1}$.  So we find again that $L_a(\CC_2) = R_1$ and $L_a(\CC_1) = R_0$ are the irreducible representations.
\end{itemize}
This completes the analysis in the case $i = 0$.  Thus far, we have shown that the finite dimensional irreducible representations correspond to the claimed list when $i = 0$.  For the general case $i \neq 0$, we may apply the equivalence of categories $\omega$ given by Lemma \ref{rotate}.  Notice that  $\omega(L_{a_0, a_1, a_2}(\tau)) = L_{a_2, a_0, a_1}(\tau \otimes \CC_2)$, so by applying $\omega$ either once or twice, we may determine the simple objects of $\OO(H_{-m_0-m_1-1, m_0+1/2, m_1+1/2}(W))$ and $\OO(H_{m_1 + 1/2, -m_0 - m_1 - 1, m_0 + 1/2}(W))$ from the simple objects of $\OO(H_{m_0 + 1/2, m_1 + 1/2, -m_0 - m_1 - 1}(W))$.  This means that $i$ applications of $\omega^{-1}$ to the finite dimensional irreducible representations for $i = 0$ gives the finite dimensional irreducible representations at general $i$.  The desired result then follows because $\omega$ preserves resolutions by exactness and it acts by $\omega(M_{a_0, a_1, a_2}(\CC_i)) = M_{a_2,a_0, a_1}(\CC_{i-1})$, $\omega(M_{a_0,a_1,a_2}(V_i)) = M_{a_2, a_0, a_1}(V_{i-1})$, and $\omega(M_{a_0, a_1, a_2}(\CC^3)) = M_{a_2, a_0, a_1}(\CC^3)$.

\noindent \textbf{Points of type (1122):} At these points, we have that $a_i = m_i + 1/2$ and $a_{i+1} = m_{i+1} + 1/2$ for some $i$ with $m_i, m_{i+1} \neq 0 \pmod{3}$.  This means that $a_{i+2} = -m_i - m_{i+1} -1$, so $a_{i+1} - a_{i+2} = m_i + 2m_{i+1} + 3/2$ and $a_{i+2} - a_i = -2m_i - m_{i+1} - 3/2$.  Hence, we see that $a$ is a point of type (1122) if and only if $m_i \equiv m_{i+1} \equiv 1, 2 \pmod{3}$.  

We first consider the case $i = 0$.  We have that $(a_0, a_1, a_2) = (m_0 + 1/2, m_1 + 1/2, -m_0 - m_1 - 1)$ and $(y, w) = (\zeta^{2m_0 + 1}, \zeta^{2m_1 + 1})$.  For integer $m_0$ and $m_1$, this gives a finite number of values of $(y, w)$.  Computing in Magma, we find that $A$ has a three dimensional nullspace spanned by $M_1 = (1,0,1,0,-1,0,0)$, $M_0=(0,1,1,-1,0,0,0)$, $N_1 = (0,1,0,0,1,0,-1)$, and $N_0 = (1,0,0,1,0,0,-1)$.  Therefore, there are exactly three finite dimensional irreducible representations at these points.

Now, by Proposition \ref{exist}, there are finite dimensional representations $R_1, R_0, S_1$, and $S_0$ with respective characters $\chi_{M_1}, \chi_{M_0}, \chi_{N_1}$, and $\chi_{N_0}$ at these points.  Let the lowest weights of these representations be $\tau_1, \tau_0, \rho_1$, and $\rho_0$.  Here, the values of $\tau_1 \in \{\CC_0, \CC_2\}$, $\tau_0 \in \{\CC_1, \CC_2\}$, $\rho_1 \in \{\CC_1, V_1\}$, and $\rho_0 \in \{\CC_0, V_0\}$ depend on the location of $a$.  We divide now into cases depending on the signs of $m_0$ and $m_1$.

\begin{itemize}
\item Suppose first that $m_0, m_1 \geq 0$.  In this case, we find that $(\tau_1, \tau_0, \rho_1, \rho_0) = (\CC_0, \CC_2, V_1, \CC_0)$ and that $h(\CC_2) > h(V_1) > h(\CC_0)$.  Thus, we see that $L_a(\CC_0), L_a(\CC_2)$, and $L_a(\CC_0)$ are the finite dimensional irreducible representations.  Notice now that the lowest weight $\CC_2$ of $R_0$ is the highest among $R_1, R_0, S_1, S_0$, so by Lemma \ref{irred}, we see that $R_0$ is irreducible, meaning that $L_a(\CC_2) = R_0$.

Now, notice that $h(\CC_2) - h(\CC_0) = 4m_0 + 2$ and that $h(V_1) - h(\CC_0) = 2m_0 + 1$.  By the character formulas for $R_1, R_0, S_1$, and $S_0$, this means that $\grade(R_1) = [h(\CC_0), h(\CC_0) + 4m_0 + 1]$, $\grade(R_0) = [h(\CC_0) + 4m_0 + 2, h(\CC_0) + 4m_0 + 4 m_1 + 3]$, $\grade(S_1) = [h(\CC_0) + 2m_0 + 2, h(\CC_0) + 4m_1 + 4m_0 + 3]$, and $\grade(S_0) = [h(\CC_0), h(\CC_0) + 4m_0 + 2 m_1 + 1]$.  Therefore, the highest weight of $R_1$ is least among $R_1, R_0, S_1$, and $S_0$; by Claim \ref{irred}, this mean that $R_1$ is irreducible and hence $L_a(\CC_0) = R_1$.

Note that $\ker(S_0 \to R_1)$ is a finite dimensional representation with character $\chi_{N_0} - \chi_{M_1} = \chi_{(0,0,-1,1,1,0,-1)}$.  Because $h(V_1) < h(V_0)$, we see that $\ker(S_0 \to R_1)$ has lowest weight $V_1$.  Now, notice that $L_a(V_1)$ is a quotient of $\ker(S_0 \to R_1)$ such that $\chi_{V_1} = \chi_{N_0} - k \chi_{M_1}$ for some $k \in \ZZ$.  Here, the coefficient of $\chi_{N_0}$ must be $1$ because $n_{V_1 V_1} = 1$, while $k \geq 1$ because $\dim L_a(V_1) \leq \dim S_0/R_1$.  But now notice that the $t^{h(\CC_0) + 4m_0 + 4m_1 + 2}$ coefficients of $\chi_{N_0}(1, t)$ and $\chi_{M_1}(1, t)$ are both equal to $1$, so $k \leq 1$ and hence $k = 1$.  Therefore, we see that $L_c(V_1) = \ker(S_0 \to R_1)$ and $\chi_{V_1} = \chi_{N_0} - \chi_{M_1}$.  In this case, we see that the finite dimensional irreducible representations are $R_0$, $R_1$, and $\ker(S_0 \to R_1)$, with characters $\chi_{M_0}$, $\chi_{M_1}$, and $\chi_{N_0 - M_1}$.

\item Now, suppose that $m_0, m_1 < 0$.  We have that $(\tau_1, \tau_0, \rho_1, \rho_0) = (\CC_2, \CC_1, \CC_1, V_0)$ and that $h(\CC_2) > h(V_0) > h(\CC_1)$.  As in the previous case, this implies that $L_a(\CC_2), L_a(V_0)$, and $L_a(\CC_1)$ are the finite dimensional irreducible representations and that $L_a(\CC_2) = R_1$. Now, notice that $h(\CC_2) - h(\CC_1) = -4m_1 + 2$ and that $h(V_0) - h(\CC_1) = -2m_1 + 1$.  As before, we conclude from the characters of $R_1, R_0, S_1$, and $S_0$ that $\grade(R_1) = [h(\CC_1) -4m_1 - 2, h(\CC_1) -4m_0 - 4m_1 - 6]$, $\grade(R_0) = [h(\CC_1), h(\CC_1) - 4m_1 - 4]$, $\grade(S_1) = [h(\CC_1), h(\CC_1) - 2m_0 - 4m_1 - 5]$, and $\grade(S_0) = [h(\CC_1) - 2m_1 - 1, h(\CC_1) - 4m_0 - 4m_1 - 6]$.  Thus we see that the highest weight of $R_0$ is least among $R_1, R_0, S_1$, and $S_0$, meaning that $L_a(\CC_1) = R_0$ by Claim \ref{irred}.

Now, note that $\ker(S_1 \to R_0)$ is a finite dimensional representation with character $\chi_{N_1} - \chi_{M_0} = \chi_{(0,0,-1,1,1,0,-1)}$. Notice that $\ker(S_1 \to R_0)$ has lowest weight $V_0$; therefore, $L_a(V_0)$ is a quotient of $\ker(S_1 \to R_0)$.  On the other hand, we may write $\chi_{V_0} = \chi_{N_1} - k \chi_{M_0}$ for some positive integer $k$; here $k$ must be positive because $\dim L_a(V_0) \leq \dim \ker(S_1 \to R_0) < \dim R_0$. But on the other hand, the coefficients of $t^{h(\CC_1) - 4 m_0 - 4 m_1 - 6}$ in $\chi_{M_0}$ and $\chi_{N_1}$ are both $1$, meaning that $k \leq 1$.  Thus we may conclude that $L_a(V_0) = \ker(S_1 \to R_0)$ and hence $\chi_{V_0} = \chi_{N_1} - \chi_{M_0}$.  In this case, we see that the irreducible representations are $R_0$, $R_1$, and $\ker(S_1 \to R_0)$, with characters $\chi_{M_0}, \chi_{M_1}$, and $\chi_{N_1 - M_0}$.  

\item Now, suppose that $m_0 \geq 0, m_1 < 0$. Then, we see that $\tau_1 = \CC_0$, $\tau_0 = \CC_1$, and
\[
\rho_1 = \begin{cases} V_1 & m_0 + 2 m_1 + 3/2 > 0 \\ \CC_1 & m_0 + 2 m_1 + 3/2 < 0 \end{cases} \text{ and } \rho_0 =  \begin{cases} \CC_0 & 2m_0 + m_1 + 3/2 > 0 \\ V_0 & 2m_0 + m_1 + 3/2 < 0 \end{cases}.
\]
We divide further into cases.

\noindent \textbf{Case 1:} If $m_0 + 2m_1 + 3/2 > 0$, then we see that $(\tau_1, \tau_0, \rho_1, \rho_0) = (\CC_0, \CC_1, V_1, \CC_0)$ and $h(\CC_1) > h(V_1) > h(\CC_0)$.  By Lemma \ref{irred}, this means that $L_a(\CC_1) = R_0$.  Now, the characters of $R_1, R_0, S_1, S_0$ imply that $\grade(R_1) = [h(\CC_0),h(\CC_0) + 4m_0]$, $\grade(R_0) = [h(\CC_0) + 4m_0 + 4m_1 + 4,h(\CC_0) + 4m_0]$, $\grade(S_1) = [h(\CC_0) + 2m_0 + 1, h(\CC_0) + 4m_0 + 4m_1 + 2]$, and $\grade(S_0) = [h(\CC_0), h(\CC_0) + 4m_0 + 2m_1 + 1]$.  Therefore, $S_1$ has the lowest highest weight of $R_1, R_0, S_1, S_0$, so by Lemma \ref{irred} we have $L_a(V_1) = S_1$.  

Now, notice that $\delta(R_1)$ and $\delta(R_0)$ have the same lowest weight, so there is a surjection $\delta(R_1) \onto \delta(R_0)$ which induces an injection $R_0 \into R_1$.  Hence, we see that $R_1/R_0$ is a finite dimensional representation with character $\chi_{M_1} - \chi_{M_0}$ and lowest weight $\CC_0$, meaning that $L_a(\CC_0)$ is a quotient of $R_1/R_0$.  Therefore, we may write $\chi_{L_a(\CC_0)} = \chi_{M_1} - \chi_{M_0} - l \chi_{M_0} - k \chi_{N_1}$ because $\chi_{M_1}, \chi_{M_0}, \chi_{N_1}$ form a basis for characters of finite dimensional representations of $H_a(W)$.  But the dimension of $L_a(\CC_0)$ in grade $4m_0$ must be non-negative, so $l = 0$.  Thus, we see that $\chi_{L_a(\CC_0)} = \chi_{(1, -1-k, 0, 1, -1-k, 0, k)}$ for some $k$; because $L_a(\CC_0)$ is a quotient of $R_1/R_0$, we have that $k \geq 0$.  By Corollary \ref{dimineq}, we have 
\[
\dim \Hom_{H_a(W)}(M_a(\CC_1), M_a(\CC_0)) \leq \dim \Hom_{\HH(\{q_{H, j}^*\})}(\widehat{\CC}_1, \widehat{\CC}_0) \leq 1.
\]
This combined with the fact that $M_a(\CC_1)$ has no self-extensions means that there can be at most one singular copy of $\CC_1$ in $M_a(\CC_0)$.  Hence, we find that $k = 0$ and therefore $\chi_{L_a(\CC_0)} = \chi_{1, -1, 0, 1, -1, 0, 0}$, meaning that $L_a(\CC_0) = R_1/R_0$.  We may now conclude that the finite dimensional irreducible representations in this case are $R_0, S_1, R_1/R_0$ with characters $\chi_{M_0}$, $\chi_{N_1}$, and $\chi_{M_1 - M_0}$, respectively.

\noindent \textbf{Case 2:} If $m_0 + 2m_1 + 3/2 < 0$ and $2m_0 + m_1 + 3/2 > 0$, we see that $(\tau_1, \tau_0, \rho_1, \rho_0) = (\CC_0, \CC_1, \CC_1, \CC_0)$ with $\grade(R_1) = [h(\CC_0), h(\CC_0) + 4m_0], \grade(R_0) = [h(\CC_0) + 4m_0 + 4m_1 + 4, h(\CC_0) + 4m_0], \grade(S_1) = [h(\CC_0) + 4m_0 + 4m_1 + 4, h(\CC_0) + 2m_0 - 1]$, and $\grade(S_0) = [h(\CC_0), h(\CC_0) + 4m_0 + 2m_1 + 1]$. 

If $m_0 + m_1 + 1> 0$, then we have $h(\CC_1) > h(\CC_0)$.  Hence, we conclude that $S_1$ has the lowest highest weight and is irreducible by Lemma \ref{irred}, so $L_a(\CC_1) = S_1$.  Because $R_0$ has lowest weight $\CC_1$, we may consider the representation $\ker(R_0 \to S_1)$; notice that it has character $\chi_{M_0} - \chi_{N_1}$.  Then, applying Lemma \ref{irred} to $R_0/S_1$ and the basis $\{M_0 - N_1, N_1, M_1\}$ shows that $\ker(R_0 \to S_1)$ is irreducible.  Now, notice that $S_0$ has unique highest weight among $\{R_1, R_0, S_1, S_0\}$, hence it has an irreducible quotient $L$ with the same highest weight.  Because $S_1$ has the only lower highest weight, we have $\chi_L = \chi_{N_0 - k N_1} = \chi_{(1, -k, 0, 1, -k, 0, -1 + k)}$ for some $k \geq 0$.  But by Lemma \ref{mappings}, there are no singular copies of $\CC_1$ in $M_a(\CC_0)$, meaning that $k = 0$, hence $L = S_1$.  Therefore, the finite dimensional irreducible representations are $S_1$, $\ker(R_0 \to S_1)$, and $S_0$ with characters $\chi_{(0, 1, 0, 0, 1, 0, -1)}$, $\chi_{(0, 0, 1, -1, -1, 0, 1)}$, and $\chi_{(1, 0, 0, 1, 0, 0, -1)}$.

If instead $m_0 + m_1 + 1 < 0$, then we have $h(\CC_1) < h(\CC_0)$.  Hence, $S_0$ has the lowest highest weight and is irreducible by Lemma \ref{irred}.  Because $R_1$ has lowest weight $\CC_0$, we may consider the representation $\ker(R_1 \to S_0)$; notice that it has character $\chi_{M_1} - \chi_{N_0}$.  Applying Lemma \ref{irred} to $\ker(R_1 \to S_0)$ and the basis $\{M_1 - N_0, N_0, M_0\}$ shows that $\ker(R_1 \to S_0)$ is irreducible.  Now, notice that $S_1$ has unique highest weight among $\{R_1, R_0, S_1, S_0\}$, so it has an irreducible quotient $L$ with the same highest weight.  Because $S_0$ has the only lower highest weight, we see that $\chi_L = \chi_{N_1 - k N_0} = \chi_{(-k, 1, 0, -k, 1, 0, -1 + k)}$ for some $k \geq 0$.  But by Lemma \ref{mappings}, there are no singular copies of $\CC_0$ in $M_a(\CC_1)$, implying that $k = 0$, hence $L = S_0$.  Thus, the finite dimensional irreducible representations are $S_0$, $\ker(R_1 \to S_0)$, and $S_1$ with characters $\chi_{(1, 0, 0, 1, 0, 0, -1)}$, $\chi_{(0, 0, 1, -1, -1, 0, 1)}$, and $\chi_{(0, 1, 0, 0, 1, 0, -1)}$.

\noindent \textbf{Case 3:} If $2m_0 + m_1 + 3/2 < 0$, we see that $(\tau_1, \tau_0, \rho_1, \rho_0) = (\CC_0, \CC_1, \CC_1, V_0)$ with $h(\CC_1) < h(\CC_0)$ and $\grade(R_1) = [h(\CC_1) + -4m_0 - 4m_1 - 4, h(\CC_1) - 4m_1 - 4]$, $\grade(R_0) = [h(\CC_1), h(\CC_1) - 4m_1 - 4]$, $\grade(S_1) = [h(\CC_1), h(\CC_1) - 2m_0 - 4m_1 - 5]$, and $\grade(S_0) = [h(\CC_1) - 2m_1 - 1, h(\CC_1) - 4m_0 - 4m_1 - 6]$.  Because $R_1$ and $S_0$ have highest lowest weight and lowest highest weight, they are irreducible.  Now, notice that $R_1$ and $R_0$ have the same highest weight, meaning that $\delta(R_1)$ and $\delta(R_0)$ have the same lowest weight.  Therefore, we have a surjection $\delta(R_0) \onto \delta(R_1)$ which induces an injection $R_1 \into R_0$.  Then, note that $R_0/R_1$ is a representation of lowest weight $\CC_1$ and character $\chi_{M_0 - M_1} = \chi_{N_1 - N_0}$.  Hence, $L_a(\CC_1)$ is a quotient of $R_0/R_1$, so it has character $\chi_{L_a(\CC_1)} = \chi_{R_0/R_1 - k N_0} = \chi_{(-1 - k, 1, 0, -1 - k, 1, k)}$ for some $k \geq 0$.  Now, by Corollary \ref{dimineq}, we have
\[
\dim \Hom_{H_a(W)}(M_a(\CC_0), M_a(\CC_1)) \leq \dim \Hom_{\HH(\{q_{H, j}^*\})}(\widehat{\CC}_0, \widehat{\CC}_1) \leq 1.
\]
Combined with the fact that $M_a(\CC_0)$ has no self-extensions, this means that there is at most one singular copy of $\CC_0$ in $M_a(\CC_1)$, hence $k = 0$.  This means that $L_a(\CC_1) = R_1/R_0$.  Therefore, the finite dimensional irreducible representations are $R_1$, $S_0$, and $R_0/R_1$ with characters $\chi_{(1, 0, 1, 0, -1, 0, 0)}$, $\chi_{(1, 0, 0, 1, 0, 0, -1)}$, and $\chi_{(1, -1, 0, 1, -1, 0, 0)}$.

\item Now, suppose that $m_0 < 0$, $m_1 \geq 0$.  The map $\delta$ of Lemma \ref{flip} gives a map between irreducible representations of $H_a(W)$ and $H_{a'}(W)$, where $a' = (m_1 + 1/2, m_0 + 1/2, -m_0 - m_1 - 1) = (m_0' + 1/2, m_1' = 1/2, -m_0' - m_1' - 1)$, where $m_0' = m_1$ and $m_1' = m_0$.  In particular, by considering limits along the lines incident to $a$, we find that $\delta(R_1) = R_0'$, $\delta(R_0) = R_1'$, $\delta(S_1) = S_0'$, and $\delta(S_0) = S_1'$, where $R_1', R_0', S_1'$, and $S_0'$ are the representations of $H_{a'}(W)$ on the lines $a_0 = m_0' + 1/2$, $a_1 = m_1' + 1/2$, $a_1 - a_2 = m_0' + 2m_1' + 3/2$, and $a_0 - a_2 = 2m_0' + m_1' + 3/2$, respectively.  We again divide into cases.

\noindent \textbf{Case 1:} If $2 m_0 + m_1 + 3/2 > 0$, then we have that $m_0' + 2 m_1' + 3/2 > 0$.  Therefore, by the previous case, $\delta(R_1)$, $\delta(S_0)$ and $\coker(\delta(R_0) \to \delta(R_1))$ are irreducible representations of $H_{a'}(W)$.  Applying $\delta$, this means that the irreducible representations of $H_a(W)$ are $R_1$, $S_0$, and $\ker(R_1 \to R_0)$, with characters $\chi_{(1, 0, 1, 0, -1, 0, 0)}$, $\chi_{(1, 0, 0, 1, 0, 0, -1)}$, and $\chi_{(-1, 1, 0, -1, 1, 0, 0)}$. 

\noindent \textbf{Case 2:} If $2 m_0 + m_1 + 3/2 < 0, m_0 + 2m_1 + 3/2 > 0, m_0 + m_1 + 1 > 0$, then $m_0' + 2 m_1' + 3/2 < 0, 2m_0' + m_1' + 3/2 > 0, m_0' + m_1' + 1 > 0$.  In this case, we see that $\delta(S_0)$, $\delta(S_1)$, and $\ker(\delta(R_1) \to \delta(S_0))$ are the irreducible representations of $H_{a'}(W)$.  Applying $\delta$ and using exactness, we see that the irreducible representations of $H_a(W)$ are $S_0$, $S_1$, and $\coker(S_0 \to R_1)$, with characters $\chi_{(1, 0, 0, 1, 0, 0, -1)}$, $\chi_{(0, 1, 0, 0, 1, 0, -1)}$, and $\chi_{(0, 0, 1, -1, 1, 0, 1)}$. 

If $2 m_0 + m_1 + 3/2 < 0, m_0 + 2m_1 + 3/2 > 0, m_0 + m_1 + 1 < 0$, then $m_0' + 2 m_1' + 3/2 < 0, 2m_0' + m_1' + 3/2 > 0, m_0' + m_1' + 1 < 0$.  In this case, we see that $\delta(S_0)$, $\delta(S_1)$, and $\ker(\delta(R_0) \to \delta(S_1))$ are irreducible representations of $H_{a'}(W)$.  Applying $\delta$, the irreducible representations of $H_a(W)$ are $S_0$, $S_1$, and $\coker(S_1 \to R_0)$, with characters $\chi_{(1, 0, 0, 1, 0, 0, -1)}$, $\chi_{(0, 1, 0, 0, 1, 0, -1)}$, and $\chi_{(0, 0, 1, -1, 1, 0, 1)}$. 

\noindent \textbf{Case 3:} If $m_0 + 2m_1 + 3/2 < 0$, then $2m_0' + m_1' + 3/2 < 0$, which means that $\delta(R_0), \delta(S_1)$, and $\coker(\delta(R_0) \to \delta(R_1))$ are irreducible representations of $H_{a'}(W)$.  Applying $\delta$, the irreducible representations of $H_a(W)$ are $R_0$, $S_1$, and $\ker(R_1 \to R_0)$, with characters $\chi_{(0, 1, 1, -1, 0, 0, 0)}$, $\chi_{(0, 1, 0, 0, 1, 0, -1)}$, and $\chi_{(1, -1, 0, 1, -1, 0, 0)}$. 
\end{itemize}
This completes the analysis in the case $i = 0$.  Thus far, we have shown that the finite dimensional irreducible representations correspond to the claimed list when $i = 0$. For $i \neq 0$, the equivalence $\omega$ given by Lemma \ref{rotate} satisfies $\omega(L_{a_0, a_1, a_2}(\tau)) = L_{a_2, a_0, a_1}(\tau \otimes \CC_2)$.  Therefore, for general $i \neq 0$, we may obtain all irreducible representations of $H_a(W)$ by applying $\omega$ either once or twice to the representations found for $i = 0$.  Combining this with the fact that $\omega(M_{a_0, a_1, a_2}(\CC_i)) = M_{a_2, a_0, a_1}(\CC_{i-1})$, $\omega(M_{a_0, a_1, a_2}(V_i)) = M_{a_2, a_0, a_1}(V_{i-1})$, and $\omega(M_{a_0, a_1, a_2}(\CC^3)) = M_{a_2, a_0, a_1}(\CC^3)$ completes the proof.
\end{proof}

\section*{Acknowledgments}

This research was done with the support of a 2009 Herchel Smith Harvard Undergraduate Summer Research Fellowship.  The author would like to thank Pavel Etingof for suggesting this problem and for many helpful conversations.

\end{document}